\documentclass[12 pt,oneside,reqno]{article}
\usepackage[english]{babel}

\usepackage[letterpaper,top=2cm,bottom=2cm,left=3cm,right=3cm,marginparwidth=1.75cm]{geometry}

\usepackage[dvips]{graphicx}
\usepackage{calc}
\usepackage{color}
\usepackage{amsmath}
\usepackage{amssymb}
\usepackage{amscd}
\usepackage{amsthm}
\usepackage{amsbsy}
\usepackage{delarray}
\usepackage{tikz}
\usepackage{tikz-cd}
\usepackage{enumerate}
\usepackage{enumitem}
\usepackage[T1]{fontenc}
\usepackage{enumerate}
\usepackage{stackengine} 
\usepackage{mathrsfs}
\usepackage{hyperref}
\usepackage[whole]{bxcjkjatype}
\usepackage{titlesec}
\usepackage[nottoc]{tocbibind}
\usepackage{dynkin-diagrams}
\usepackage{etoolbox}
\usepackage{dynkin-diagrams}
\def\row#1/#2!{#1_{\IfStrEq{#2}{}{n}{#2}} & \dynkin{#1}{#2}\\}

\usepackage{lipsum}

\definecolor{YK}{rgb}{9,0,0}
\definecolor{YKb}{rgb}{0,0,9}

 \date{}

\begin{document}



\setlength{\parindent}{5mm}
\renewcommand{\leq}{\leqslant}
\renewcommand{\geq}{\geqslant}
\newcommand{\N}{\mathbb{N}}
\newcommand{\sph}{\mathbb{S}}
\newcommand{\Z}{\mathbb{Z}}
\newcommand{\R}{\mathbb{R}}
\newcommand{\Q}{\mathbb{Q}}
\newcommand{\C}{\mathbb{C}}
\newcommand{\curlL}{\mathcal{L}}
\newcommand{\A}{\mathcal{A}}
\newcommand{\curlP}{\mathcal{P}}
\newcommand{\F}{\mathbb{F}}
\newcommand{\g}{\mathfrak{g}}
\newcommand{\h}{\mathfrak{h}}
\newcommand{\K}{\mathbb{K}}
\newcommand{\RN}{\mathbb{R}^{2n}}
\newcommand{\ci}{c^{\infty}}
\newcommand{\derive}[2]{\frac{\partial{#1}}{\partial{#2}}}
\renewcommand{\S}{\mathbb{S}}
\renewcommand{\H}{\mathbb{H}}
\newcommand{\eps}{\varepsilon}
\newcommand{\lag}{Lagrangian}
\newcommand{\sub}{submanifold}
\newcommand{\homo}{homogeneous}
\newcommand{\qmor}{quasimorphism}
\newcommand{\enum}{enumerate}
\newcommand{\sa}{symplectically aspherical}
\newcommand{\ovl}{\overline}
\newcommand{\wt}{\widetilde}
\newcommand{\hamd}{Hamiltonian diffeomorphism}
\newcommand{\QH}{quantum homology ring}
\newcommand{\thm}{Theorem}
\newcommand{\cor}{Corollary}
\newcommand{\hamil}{Hamiltonian}
\newcommand{\propo}{Proposition}
\newcommand{\conjec}{Conjecture}
\newcommand{\asympt}{asymptotic}
\newcommand{\PR}{pseudo-rotation}
\newcommand{\fuk}{\mathscr{F}}
\newcommand{\specinv}{spectral invariant}
\newcommand{\wrt}{with respect to}
\newcommand{\suphv}{superheavy}
\newcommand{\suphvness}{superheaviness}
\newcommand{\symp}{symplectic}
\newcommand\oast{\stackMath\mathbin{\stackinset{c}{0ex}{c}{0ex}{\ast}{\bigcirc}}}
\newcommand{\diffeo}{diffeomorphism}
\newcommand{\quasi}{quasimorphism}
\newcommand{\univham}{\widetilde{\Ham}}
\newcommand{\BK}{Biran--Khanevsky}
\newcommand{\BC}{Biran--Cornea}
\newcommand{\TV}{Tonkonog--Varolgunes}
\newcommand{\varol}{Varolgunes}
\newcommand{\mfd}{manifold}
\newcommand{\smfd}{submanifold}
\newcommand{\EP}{Entov--Polterovich}
\newcommand{\pol}{Polterovich}
\newcommand{\suppot}{superpotential}
\newcommand{\CY}{Calabi--Yau}
\newcommand{\degen}{degenerate}
\newcommand{\coeff}{coefficient}
\newcommand{\acs}{almost complex structure}
\newcommand{\prl}{pearl trajectory}
\newcommand{\prls}{pearl trajectories}
\newcommand{\holo}{holomorphic}
\newcommand{\trans}{transversality}
\newcommand{\pert}{perturbation}
\newcommand{\idem}{idempotent}
\newcommand{\critpt}{critical point}
\newcommand{\tordeg}{toric degeneration}
\newcommand{\nbhd}{neighborhood}
\newcommand{\polar}{polarization}
\newcommand{\symplecto}{symplectomorphism}
\newcommand{\FOOO}{Fukaya--Oh--Ohta--Ono}
\newcommand{\NNU}{Nishinou--Nohara--Ueda}

\theoremstyle{plain}
\newtheorem{theo}{Theorem}
\newtheorem{theox}{Theorem}
\renewcommand{\thetheox}{\Alph{theox}}
\numberwithin{theo}{subsection}
\newtheorem{prop}[theo]{Proposition}
\newtheorem{lemma}[theo]{Lemma}
\newtheorem{definition}[theo]{Definition}
\newtheorem*{notation*}{Notation}
\newtheorem*{notations*}{Notations}
\newtheorem{corol}[theo]{Corollary}
\newtheorem{conj}[theo]{Conjecture}
\newtheorem{guess}[theo]{Guess}
\newtheorem{claim}[theo]{Claim}
\newtheorem{question}[theo]{Question}
\newtheorem{prob}[theo]{Problem}
\numberwithin{equation}{subsection}

\newenvironment{demo}[1][]{\addvspace{8mm} \emph{Proof #1.
    ~~}}{~~~$\Box$\bigskip}

\newlength{\espaceavantspecialthm}
\newlength{\espaceapresspecialthm}
\setlength{\espaceavantspecialthm}{\topsep} \setlength{\espaceapresspecialthm}{\topsep}

\newenvironment{example}[1][]{\refstepcounter{theo} 
\vskip \espaceavantspecialthm \noindent \textsc{Example~\thetheo
#1.} }%
{\vskip \espaceapresspecialthm}

\newenvironment{remark}[1][]{\refstepcounter{theo} 
\vskip \espaceavantspecialthm \noindent \textsc{Remark~\thetheo
#1.} }%
{\vskip \espaceapresspecialthm}

\def\bb#1{\mathbb{#1}} \def\m#1{\mathcal{#1}}

\def\momeg{(M,\omega)}
\def\co{\colon\thinspace}
\def\Homeo{\mathrm{Homeo}}
\def\Diffeo{\mathrm{Diffeo}}
\def\Symp{\mathrm{Symp}}
\def\Sympeo{\mathrm{Sympeo}}
\def\id{\mathrm{id}}
\newcommand{\norm}[1]{||#1||}
\def\Ham{\mathrm{Ham}}
\def\lagham#1{\mathcal{L}^\mathrm{\Ham}({#1})}
\def\Hamtilde{\widetilde{\mathrm{\Ham}}}
\def\cOlag#1{\mathrm{Sympeo}({#1})}
\def\Crit{\mathrm{Crit}}
\def\diag{\mathrm{diag}}
\def\Spec{\mathrm{Spec}}
\def\osc{\mathrm{osc}}
\def\Cal{\mathrm{Cal}}
\def\Ker{\mathrm{Ker}}
\def\Hom{\mathrm{Hom}}
\def\FS{\mathrm{FS}}
\def\tor{\mathrm{tor}}
\def\Int{\mathrm{Int}}
\def\PD{\mathrm{PD}}
\def\Spec{\mathrm{Spec}}
\def\momeg{(M,\omega)}
\def\co{\colon\thinspace}
\def\Homeo{\mathrm{Homeo}}
\def\Hameo{\mathrm{Hameo}}
\def\Diffeo{\mathrm{Diffeo}}
\def\Symp{\mathrm{Symp}}
\def\Sympeo{\mathrm{Sympeo}}
\def\id{\mathrm{id}}
\def\Im{\mathrm{Im}}
\def\Ham{\mathrm{Ham}}
\def\lagham#1{\mathcal{L}^\mathrm{Ham}({#1})}
\def\Hamtilde{\widetilde{\mathrm{Ham}}}
\def\cOlag#1{\mathrm{Sympeo}({#1})}
\def\Crit{\mathrm{Crit}}
\def\dim{\mathrm{dim}}
\def\Spec{\mathrm{Spec}}
\def\osc{\mathrm{osc}}
\def\Cal{\mathrm{Cal}}
\def\Fix{\mathrm{Fix}}
\def\det{\mathrm{det}}
\def\Ker{\mathrm{Ker}}
\def\coker{\mathrm{coker}}
\def\Per{\mathrm{Per}}
\def\rank{\mathrm{rank}}
\def\Span{\mathrm{Span}}
\def\Supp{\mathrm{Supp}}
\def\Hof{\mathrm{Hof}}
\def\can{\mathrm{can}}
\def\grad{\mathrm{grad}}
\def\ind{\mathrm{ind}}
\def\Hor{\mathrm{Hor}}
\def\Vert{\mathrm{Vert}}
\def\Re{\mathrm{Re}}

\title{Donaldson divisors and spectral invariants}
\author{Yusuke Kawamoto}

\newcommand{\Addresses}{{
  \bigskip
  \footnotesize

   \textsc{Yusuke Kawamoto, Institute for Mathematical Research (FIM), ETH-Z\"urich, R\"amistrasse 101, 8092 Z\"urich Switzerland}\par\nopagebreak
  \textit{E-mail address}: \texttt{yusukekawamoto81@gmail.com, yusuke.kawamoto@math.ethz.ch} }}

\maketitle

\begin{abstract}
We establish a comparison between spectral invariants for a symplectic manifold and a Donaldson divisor therein, and answer a question of Borman from 2012 on the reduction of Entov--Polterovich quasimorphisms, under a reasonable assumption. The method involves a quantitative interpretation of Biran--Khanevsky's quantum Gysin sequence.

\end{abstract}

\tableofcontents

\section{Introduction}\label{intro}

\subsection{Historical context}

Since the seminal discovery of Donaldson \cite{[Don96]} in the mid 1990s, a hypersurface $\Sigma$ in a {\symp} {\mfd} $(X,\omega)$ that satisfies $PD([\Sigma])= k [\omega]$ for some integer $k$, which we will refer to as a \textit{Donaldson divisor}, became an important object in {\symp} topology and had many fruitful consequences, e.g. Seidel's proof of the Homological Mirror Symmetry conjecture for the quartic surfaces \cite{[Sei15]}. One remarkable fact about Donaldson divisors is that its complement, i.e. $X \backslash \Sigma$, has a nice symplecto-geometric structure, namely the Liouville structure.

Biran noticed in the early 2000s that not only the complement of a Donaldson divisor but also the complement of the \textit{skeleton} (which is, roughly speaking, the stable subset of the Liouville flow; see Section \ref{biran decomp} for its precise definition) has a nice symplecto-geometric structure, namely the symplectic disk bundle:
$$X \backslash \Delta \xrightarrow[]{\simeq}  D\Sigma $$
where $\Delta$ denotes the skeleton and $D\Sigma$ is a symplectic disk bundle over $\Sigma$ (see Section \ref{biran decomp}). This result is now known as the \textit{Biran decomposition} and have found applications in {\symp} embedding problems, {\lag} rigidity phenomena \cite{[Bir01],[Bir06]}.

In a different direction, around the same time as Biran's discovery, {\EP} \cite{[EP03]} brought a new insight to \textit{Hofer geometry}, i.e. the study of the geometry of the group of {\hamil} {\diffeo}s $\Ham (X)$\footnote{Its universal lift is denoted by $\wt{\Ham} (X)$).}, which is a fundamental object of study in {\symp} topology. {\EP} constructed \textit{quasimorphisms} on $\Ham (X)$ under some condition on the quantum cohomology ring (see Section \ref{EP qmors and superheaviness}), and this discovery triggered an extensive study of {\qmor}s in {\symp} topology, which is summarized in Entov's ICM-address \cite{[Ent14]}.

Borman \cite{[Bor12]} made an interesting observation that given a Donaldson divisor $\Sigma$ in $X$, {\qmor}s on $\wt{\Ham} (X)$ get pulled-back to {\qmor}s on $\wt{\Ham} (\Sigma)$ provided that the skeleton $\Delta$ is \textit{small} (see Section \ref{Borman's reduction result} for the precise meaning of smallness), i.e. there is a map
\begin{equation}\label{borman map}
    \Theta^{\ast}: \{\mu: \wt{\Ham} (X) \to \R : \text{{\qmor}} \} \to \{\mu: \wt{\Ham} (\Sigma) \to \R : \text{{\qmor}} \} .
\end{equation}

Note that Borman's result concerns not only {\qmor}s of the {\EP}-type but all {\qmor}s on $\wt{\Ham} (X)$. However, it was not clear if {\EP}-type {\qmor}s on $\wt{\Ham} (X)$ get pulled-back to {\EP}-type {\qmor}s on $\wt{\Ham} (\Sigma)$:

\begin{question}[Borman's question, \cite{[Bor12],[Bor13]}]\label{Borman's question}
Let $(X,\Sigma)$ be a pair of a {\symp} {\mfd} and a Donaldson divisor\footnote{Borman defines Donaldson divisors slighly differently; see Definition \ref{donaldson}}. Suppose that there is an {\EP} {\qmor} for $X$
$$\mu_{X} ^{EP} : \wt{\Ham} (X) \longrightarrow \R .$$
Is the pulled-back {\qmor} 
$$\Theta^{\ast} \mu_{X} ^{EP} : \wt{\Ham} (\Sigma) \longrightarrow \R $$
an {\EP} {\qmor}, i.e. does there exist an {\EP} {\qmor} for $\Sigma$ 
$$\mu_{\Sigma} ^{EP} : \wt{\Ham} (\Sigma) \longrightarrow \R $$
such that 
$$\Theta^{\ast} \mu_{X} ^{EP}  =\mu_{\Sigma} ^{EP} ? $$
\end{question}

The difficulty of this question comes from the fact that there is no natural ring homomorphism between $QH(X)$ and $QH(\Sigma)$, and there is no progress on this question since it was posed in the early 2010s. Note that there is some ambiguity in the Borman's question as a priori, it is not clear whether the quantum cohomology rings of the {\symp} manifolds $X$ and $\Sigma$ satisfy the condition which guarantees the existence of {\EP} {\qmor}s on $X$ and $\Sigma$.

\subsection{Main result}

The main aim of this paper is to answer Question \ref{Borman's question}. Once again, we emphasize that in Borman's question, it is not clear whether the quantum cohomology rings of the {\symp} manifolds $X$ and $\Sigma$ satisfy the condition which allows one get an {\EP} {\qmor}. 

We now state the main result of the paper which positively answers Borman's question under a reasonable condition.

\begin{theox}\label{main thm intro}
Let $(X,\Sigma)$ be a pair of a closed monotone {\symp} {\mfd} and a Donaldson divisor. Assume that there is a monotone {\lag} torus $L$ in $\Sigma$ whose {\suppot} $W_L$ has a non-degenerate critical point as well as its lifted monotone {\lag} torus $\wt{L}$. Then, there exists an {\EP} {\qmor} for $X$
$$\mu_{X} ^{EP} : \wt{\Ham} (X) \longrightarrow \R $$
for which the skeleton $\Delta$ is small {\wrt} $\mu_{X} ^{EP}$ and the pulled-back {\qmor} $ \Theta^\ast \mu_{X} ^{EP}$ is an {\EP} {\qmor} for $\Sigma$, i.e. there exist an {\EP} {\qmor} for $\Sigma$ 
$$\mu_{\Sigma} ^{EP} : \wt{\Ham} (\Sigma) \longrightarrow \R $$
such that 
$$\Theta^{\ast} \mu_{X} ^{EP}  =\mu_{\Sigma} ^{EP} .$$

\end{theox}

\begin{remark}\label{rmk main thm}
\begin{enumerate}
    \item Theorem \ref{main thm} states the same result more precisely but with some notions from the preliminary section (Section \ref{prelim}).

    \item We believe that the condition in {\thm} \ref{main thm intro} is optimal, i.e. whenever there is an {\EP} {\qmor}s on $X$ for which the skeleton $\Delta$ is small, there are monotone {\lag} tori $L$ and $\wt{L}$ as in Theorem \ref{main thm intro}, c.f. \cite[{\conjec} 1.1]{[Aur07]}. See Section \ref{Discussion} for a discussion on this.

    \item It is possible that the second condition in Theorem \ref{main thm intro} follows automatically from the first condition by the work of Diogo--Tonkonog--Vianna--Wu \cite{[DTVW]}. See Section \ref{Discussion} for further remarks.
\end{enumerate}
\end{remark}

In order to illustrate {\thm} \ref{main thm intro}, we list some examples to which Theorem \ref{main thm intro} applies. For further information and more detailed remarks concerning examples in Example \ref{example intro}, see Section \ref{examples}.

\begin{example}\label{example intro}
\begin{enumerate}
    \item $(X,\Sigma)=(\C P^n, \C P^{n-1})$: There are (unique) {\EP} {\qmor}s $\mu_{\C P^n} ^{\text{EP}}$, $\mu_{\C P^{n-1}} ^{\text{EP}}$ for $\C P^n$, $\C P^{n-1}$, respectively. The skeleton for the pair $(\C P^n, \C P^{n-1})$ is a point which satisfies the \textit{smallness condition}. Thus, Borman's reduction theorem is applicable and $\Theta^\ast \mu_{\C P^n;EP}$ is a {\qmor} for $\C P^{n-1}$, which we do not know at this point whether or not it is of {\EP}-type. The assumption of Theorem \ref{main thm intro} is satisfied for the pair of {\lag} tori $(\wt{L}=T^n _{\text{Clif}}$, $L=T^{n-1} _{\text{Clif}})$ and thus Theorem \ref{main thm intro} implies that the following holds:
    $$\Theta^\ast \mu_{\C P^n} ^{\text{EP}} = \mu_{\C P^{n-1}} ^{\text{EP}}  $$
    up to a constant factor.

    \item $(X,\Sigma)=(S^2 \times S^2, \Delta:=\{(x,x)\in S^2 \times S^2 \} \simeq S^2)$: Note that this is merely the $n=2$ case of the next example $(X,\Sigma)=(Q^n, Q^{n-1})$ but we decided that it is still instructive to treat this case differently. There are two {\EP} {\qmor}s for $S^2 \times S^2$ and only one for $\Delta$. We denote them by $\mu_{S^2 \times S^2,\pm} ^{\text{EP}}$ and $\mu_{S^2} ^{\text{EP}}$. The skeleton for the pair $(S^2 \times S^2, \Delta)$ is the anti-diagonal 
    $$ \ovl{ \Delta } :=\{(x,-x)\in S^2 \times S^2 \} $$
    which satisfies the \textit{smallness condition} for $\mu_{S^2 \times S^2,+} ^{\text{EP}}$ but not for $\mu_{S^2 \times S^2,-} ^{\text{EP}}$. Thus, Borman's reduction theorem is applicable only to $\mu_{S^2 \times S^2,+} ^{\text{EP}}$ 
    and $\Theta^\ast \mu_{S^2 \times S^2,+} ^{\text{EP}} $ is a {\qmor} for $\Delta \simeq S^2$, which we do not know at this point whether or not it is of {\EP}-type. The assumption of Theorem \ref{main thm intro} is satisfied for the pair of {\lag} tori $(\wt{L}=T^2 _{\text{Ch}}$, $L=S^1 _{\text{equator}} ) $ and thus Theorem \ref{main thm intro} implies that the following holds:
    $$\Theta^\ast \mu_{S^2 \times S^2,+} ^{\text{EP}} = \mu_{S^2} ^{\text{EP}} $$
    up to a constant factor.

    \item $(X,\Sigma)=(Q^n, Q^{n-1})$ where
    $$Q^n:=\{[z_0:z_1:\cdots :z_{n+1}] \in \C P^{n+1} | z_0 ^2 + z_1 ^2 + \cdots+ z_{n+1} ^2 =0  \} , $$
    $$ Q^{n-1}: = \{z\in Q^n: z_{n+1}=0 \}:$$
    There are two {\EP} {\qmor}s each for $Q^{n}$ and $Q^{n-1}$, which we denote by $\mu_{Q^{n},\pm} ^{\text{EP}}$ and $\mu_{Q^{n-1},\pm} ^{\text{EP}}$. The skeleton for the pair $(Q^n, Q^{n-1})$ is 
    $$ S^n  := \{z  \in Q^{n} | z_0 ^2 + \cdots+z_{n} ^2+ z_{n+1} ^2 =0 ,\ z_0,\cdots,z_n \in \R,\ z_{n+1} \in i \R \} $$
    which satisfies the \textit{smallness condition} for $\mu_{Q^{n},+} ^{\text{EP}}$ but not for $\mu_{Q^{n},-} ^{\text{EP}}$. Thus, Borman's reduction theorem is applicable only to $\mu_{Q^{n},+} ^{\text{EP}}$ 
    and $\Theta^\ast \mu_{Q^{n},+} ^{\text{EP}} $ is a {\qmor} for $Q^{n-1}$, which we do not know at this point whether or not it is of {\EP}-type. The assumption of Theorem \ref{main thm intro} is satisfied for the pair of {\lag} tori $(\wt{L}=T^n _{\text{GZ}}$, $L=T^{n-1} _{\text{GZ}})$ (see \cite[{\thm} B(1)]{Kaw23}) and thus Theorem \ref{main thm intro} implies that the following holds:
    $$\Theta^\ast \mu_{Q^{n},+} ^{\text{EP}} = \mu_{Q^{n-1},+} ^{\text{EP}} $$
    up to a constant factor.
    
    \item $(X,\Sigma)=(\C P^3, Q^{2})$: For $\C P^3$, there is a unique {\EP} {\qmor} $\mu_{\C P^3} ^{\text{EP}}$, and for $Q^{2}$, there are two {\EP} {\qmor}s, $\mu_{Q^{2},\pm} ^{\text{EP}}$. The skeleton for the pair $(\C P^3, Q^{2})$ is $ \R P^3 $ which satisfies the \textit{smallness condition} for $\mu_{\C P^3} ^{\text{EP}}$. Thus, Borman's reduction theorem is applicable to $\mu_{\C P^3} ^{\text{EP}}$ 
    and $\Theta^\ast \mu_{\C P^3} ^{\text{EP}} $ is a {\qmor} for $Q^{2}$, which we do not know at this point whether or not it is of {\EP}-type. The assumption of Theorem \ref{main thm intro} is satisfied for the pair of {\lag} tori $(\wt{L}=T^3 _{\text{Ch}}$, $L=T^{2} _{\text{Ch}})$ and thus Theorem \ref{main thm intro} implies that the following holds:
    $$\Theta^\ast \mu_{\C P^3} ^{\text{EP}} = \mu_{Q^{2},+} ^{\text{EP}} $$
    up to a constant factor.

\end{enumerate}
\end{example}

\subsection{Strategy}\label{Strategy}
We summarize our approach to prove Theorem \ref{main thm intro} and explain some ideas behind it.

As {\EP} {\qmor}s, which are reviewed in Section \ref{EP qmors and superheaviness}, are constructed by {\hamil} {\specinv}s, the main point of Question \ref{Borman's question} is to study the the relation between the {\hamil} {\specinv}s for the {\symp} {\mfd} $X$ and its Donaldson divisor $\Sigma$. The main difficulty to do that is that the Floer/quantum cohomologies of $X$ and $\Sigma$ are a priori not related, e.g. there is no natural homomorphism from one to the other.

However, {\BK} \cite{[BK13]} constructed a long exact sequence, which we call the \textit{quantum Gysin sequence}, that relates {\lag} Floer homologies of a {\lag} $L$ in the Donaldson divisor $\Sigma$ and its lifted {\lag} $\wt{L}$ in $X$:
\begin{equation}
\begin{tikzcd}
\arrow[r,""] & QH^{\ast} (L,\rho) \arrow[r,"i_{}"] & QH^{\ast} (\wt{L},\wt{\rho}) \arrow[r,"p"] & QH^{\ast-1} (L,\rho) \arrow[r,""] &.
\end{tikzcd}
\end{equation}

In Section \ref{quantum Gysin}, we will review the construction of the quantum Gysin sequence by enhancing it to a version with local systems. {\BK}'s method is based on {\BC}'s pearl theory and therefore, it is not suited to consider {\symp} invariants that are defined through filtration in Floer homology such as {\specinv}s. In order to overcome this issue, we translate their construction into the language of {\lag} Floer homology, namely we construct its Floer-theoretic counterpart which we call the \textit{Floer--Gysin sequence}. Finally, by taking the action filtration into consideration, we get the \textit{filtered Floer--Gysin sequence} (Section \ref{Filtered Floer--Gysin sequence}):
\begin{equation}\label{strategy seq}
\begin{tikzcd}
\arrow[r,""] & HF_{\Sigma} ^{\tau} (L,H) \arrow[r,"i_{Fl}"] & HF_{X} ^{h(r_0) \tau + \eps' } (\wt{L},\wt{H}) \arrow[r,"p_{Fl}"] & HF_{\Sigma} ^{\tau} (L,H) \arrow[r,""] & ,
\end{tikzcd}
\end{equation}
which is compatible with {\BK}'s quantum Gysin sequence. The reason why we do not work only with the Floer--Gysin sequence is that, some computations and arguments are easier to do with the pearl theory than with the Floer theory, c.f. Section \ref{The connecting map}.

Once we have established the filtered Floer--Gysin sequence \eqref{strategy seq}, we study the relation between the {\lag} {\specinv}s of $L$ and $\wt{L}$. We can also relate the {\lag} {\specinv}s of $L$ and $\wt{L}$ to the {\hamil} {\specinv}s of $\Sigma$ and $X$, respectively, via the closed-open/open-closed maps. The following diagram might be enlightening to understand the approach:

\begin{equation}\label{filtered diagram}
\begin{tikzcd}
  \arrow[r] & HF_{\Sigma} ^{\tau}(L, H)  \arrow[r] & HF_{X} ^{h(r_0)\cdot \tau + \eps' } (\wt{L},  \wt{H}) \arrow[shift right, swap]{d}{\mathcal{OC}^{0}} \arrow[r] & HF_{\Sigma} ^{\tau}(L, H)  \arrow[d,"\mathcal{OC}^{0}"] \arrow[r] &  \ \\
      & HF_{\Sigma} ^{\tau}( H) \arrow[u,"\mathcal{CO}^{0}"] & HF_{X} ^{h(r_0)\cdot \tau + \eps'}( \wt{H}) \arrow[shift right, swap]{u}{\mathcal{CO}^{0}}  & HF_{\Sigma} ^{\tau}( H)   &  .
\end{tikzcd}
\end{equation}

Through diagram \eqref{filtered diagram}, we can study the relation between the {\hamil} {\specinv}s of $\Sigma$ and $X$. All this is done in Section \ref{Completing the proof}.

Examples to which the main result applies are discussed in Example \ref{example intro} and Section \ref{examples} where the latter contains more details.

\subsection{Acknowledgements}
I thank Paul Biran and Octav Cornea for useful discussions, and Kaoru Ono and Leonid Polterovich for their interesting comments and feedback at the Mittag--Leffler institute in 2022. This project was conducted at Universit\'e de Montr\'eal while the author was a CRM-postdoctoral fellow at Centre de Recherches Math\'ematiques (CRM). The author thanks CRM for their hospitality.

\section{Preliminaries}\label{prelim}
We start this section by precising some conventions that are sometimes implicit in the paper. We say that a {\symp} {\mfd} $(X,\omega)$ is an \textit{integral} {\symp} {\mfd} if the {\symp} form admits an integral lift, i.e. $[\omega] \in H^2 (X ; \R)$ can be seen as an element of the integral cohomology $ H^2 (X ; \Z)$.

Unless otherwise mentioned, the {\symp} {\mfd}s and the {\lag} {\smfd}s that we consider in the results/proofs are assumed to be monotone; a {\symp} {\mfd} $(X,\omega)$ and a {\lag} {\smfd} $L$ are monotone when there exist $\tau , \kappa >0$ such that 
$$ \omega|_{\pi_2 (X)} = \tau \cdot \mu_{CZ}|_{\pi_2 (X)}  ,$$
$$ \omega|_{\pi_2 (X,L)} = \kappa \cdot  \mu_{L}|_{\pi_2 (X,L)}  $$
where $\mu_{CZ},\mu_{L}$ denote the Conley--Zehnder and Maslov indices, respectively.

\subsection{Spectral invariant theory}\label{prelim specinv}
 
It is well-known that on a closed {\symp} {\mfd} $(X,\omega)$\footnote{Although the results in this section hold for general closed {\symp} {\mfd}s, we will only be using the monotone case due to some Floer-theoretic constraints that will appear later, which is not from the {\specinv} theory.}, for a non-{\degen} {\hamil} $H:=\{H_t:X \to \R \}_{t\in [0,1]}$ and a choice of a nice {\coeff} field $\Lambda^{\downarrow}$, such as the downward Laurent {\coeff}s $\Lambda_{\text{Lau}} ^{\downarrow} $ for the monotone case
$$\Lambda_{\text{Lau}} ^{\downarrow} :=\{\sum_{k\leq k_0 } b_k t^k : k_0\in \mathbb{Z},b_k \in \mathbb{C} \}  ,   $$
or the downward Novikov {\coeff}s $\Lambda_{\text{Nov}} ^{\downarrow}$ for the general case
$$\Lambda_{\text{Nov}} ^{\downarrow}:=\{\sum_{j=1} ^{\infty} a_j T^{\lambda_j} :a_j \in \mathbb{C}, \lambda _j  \in \mathbb{R},\lim_{j\to -\infty} \lambda_j =+\infty \} ,$$
one can construct a filtered Floer homology group $\{HF^\tau(H):=HF^\tau(H;\Lambda^{\downarrow})\}_{\tau \in \R}$. Note that in this paper, we only use Novikov {\coeff}s, i.e. 
$$ \Lambda ^{\downarrow}=\Lambda_{\text{Nov}} ^{\downarrow} .$$
For two numbers $\tau<\tau'$, the groups $HF^\tau(H;\Lambda^{\downarrow})$ and $HF^{\tau'}(H;\Lambda^{\downarrow})$ are related by a map induced by the inclusion map on the chain level:
$$i_{\tau,\tau'}: HF^\tau(H;\Lambda^{\downarrow}) \longrightarrow HF^{\tau'}(H;\Lambda^{\downarrow}) ,$$
and especially we have
$$i_{\tau}: HF^\tau(H;\Lambda^{\downarrow}) \longrightarrow HF^{}(H;\Lambda^{\downarrow}) ,$$
where $HF^{}(H;\Lambda^{\downarrow})$ is the Floer homology group. There is a canonical ring isomorphism called the Piunikhin--Salamon--Schwarz (PSS)-map \cite{[PSS96]}, \cite{[MS04]}
$$PSS_{H; \Lambda_{} } : QH(X,\omega ;\Lambda_{}) \xrightarrow{\sim} HF(H;\Lambda_{} ^{\downarrow}) ,$$
where $QH(X,\omega;\Lambda_{})$ denotes the quantum cohomology ring of $(X,\omega)$ with $\Lambda$-{\coeff}s, i.e.
$$ QH(X,\omega;\Lambda_{}) := H^\ast (X;\C) \otimes \Lambda_{} .$$
Here, $\Lambda_{}$ is the Novikov {\coeff}s (the universal Novikov field) $\Lambda_{\text{Nov}}$
$$\Lambda_{\text{Nov}}:=\{\sum_{j=1} ^{\infty} a_j T^{\lambda_j} :a_j \in \mathbb{C}, \lambda _j  \in \mathbb{R},\lim_{j\to +\infty} \lambda_j =+\infty \} .$$ 
From now on, we will always take the the universal Novikov field to set-up the {\QH}, so we will often abbreviate it by $QH(X,\omega)$, i.e. 
$$QH(X,\omega) := QH(X,\omega ;\Lambda_{\text{Nov}}) .$$

The ring structure of $QH(X,\omega)$ is given by the quantum product, which is a quantum deformation of the intersection product
$$- \ast - :QH(X,\omega) \times QH(X,\omega) \to QH(X,\omega) . $$

The {\specinv}s, which were introduced by Schwarz \cite{[Sch00]} and developed by Oh \cite{[Oh05]} following the idea of Viterbo \cite{[Vit92]}, are real numbers $\{c (H,a ) \in \R\}$ associated to a pair of a {\hamil} $H$ and a class $a \in QH(X,\omega )$ in the following way:
$$c(H,a ) := \inf \{\tau \in \R : PSS_{H; \Lambda_{} } (a) \in \Im (i_{\tau})\} .$$

\begin{remark}
Although the Floer homology is only defined for a non-{\degen} {\hamil} $H$, the {\specinv}s can be defined for any {\hamil} by using the following \textit{Hofer continuity property}:
\begin{equation}\label{Hofer conti abs}
   \int_{0} ^1 \min_{x\in X} \left(H_t(x) - G_t(x) \right) dt \leq  c(H,a )-c(G,a ) \leq \int_{0} ^1 \max_{x\in X} \left(H_t(x) - G_t(x) \right) dt
\end{equation}
for any $a \in QH(X,\omega),\ H$ and $G$.
\end{remark}

Spectral invariants satisfy the \textit{triangle inequality}: for {\hamil}s $H,G$ and $a,b \in QH(X,\omega)$, we have
\begin{equation}\label{triangle ineq}
    c(H,a) + c(G, b) \geq c(H \# G, a \ast b )
\end{equation}
where $H \# G (t,x):=H_t(x) + G_t( \left( \phi^t _{H} \right)^{-1} (x)  )$ and it generates the path $t \mapsto \phi^t _H \circ \phi^t _G$ in $\Ham(X,\omega)$.

When we take the zero function as the {\hamil}, we have the \textit{valuation property}: for any $a\in  QH(X;\Lambda) \backslash \{0\},$
\begin{equation}\label{valuation}
   c (0,a)=\nu(a) 
\end{equation}
where $0$ is the zero-function and $\nu : QH(X;\Lambda) \to \R$ is the natural valuation function
\begin{equation}
    \begin{gathered}
        \nu : QH(X;\Lambda) \to \R \\
        \nu(a):= \nu(\sum_{j=1} ^{\infty} a_j T^{\lambda_j}):=\min\{\lambda_j: a_j \neq 0\}.
    \end{gathered}
\end{equation}

Note that from the triangle inequality \eqref{triangle ineq} and the valuation property \eqref{valuation}, for any $a\in  QH(X;\Lambda) \backslash \{0\},\ \lambda \in \Lambda$ and a {\hamil} $H$, we have
\begin{equation}\label{novikov shift}
    c(H, \lambda \cdot a) =  c(H,   a) + \nu (\lambda) .
\end{equation}

Analogous invariants for {\lag} Floer homology, namely the {\lag} {\specinv}s, were defined in \cite{[Lec08],[LZ18],[FOOO19],[PS]}. We summarize some basic properties of {\lag} {\specinv}s from these references. Once again, given a pair of a (non-{\degen}) {\hamil} $H$ and a class $a \in HF(L)$\footnote{The {\lag} Floer homology for $L$ without a {\hamil} term $HF(L)$ stands for the {\lag} quantum cohomology \cite{[BC09]}, which is also written as $QH(L)$ in the literature.}, we define 
$$\ell (H,\alpha ) := \inf \{\tau \in \R : PSS_{L,H } (\alpha) \in \Im (i_{\tau} ^L)\} $$
where 
$$PSS_{L,H } : HF(L) \rightarrow HF(L,H),$$
$$i_{\tau} ^L: HF^\tau (L,H) \to HF^\tau (L,H)  .$$
In this paper, we pay particular attention to the case where $\alpha= 1_L$. In this case, we simply denote 
$$\ell_L (H): = \ell (H,1_L ) .$$

Analogously to the {\hamil} case (c.f. \eqref{Hofer conti abs}), we have the \textit{{\lag} control property} for $\ell_L$:
\begin{equation}\label{Lag control prop}
   \int_{0} ^1 \min_{x\in L} H_t(x) dt \leq \ell_L (H)  \leq \int_{0} ^1 \max_{x\in L} H_t(x) dt
\end{equation}

Properties analogous to \eqref{triangle ineq}, \eqref{valuation}, \eqref{novikov shift} also hold for {\lag} {\specinv}s.

Note that both {\hamil} and {\lag} {\specinv}s satisfy the homotopy invariance, i.e. if two normalized {\hamil}s $H$ and $G$ generate homotopic {\hamil} paths $t \mapsto \phi_H ^t$ and $t \mapsto \phi_G ^t$ in $\Ham(X,\omega)$, then 
$$c(H, -) =c(G, -) . $$
Thus, one can define {\specinv}s on $\wt{\Ham} (X,\omega)$:
\begin{equation}
\begin{gathered}
 c: \wt{\Ham} (X,\omega) \times QH(X,\omega) \rightarrow \R \\
 c(\wt{\phi},a):= c(H ,a)
\end{gathered}
\end{equation}
where the path $t \mapsto \phi_H ^t$ represents the class of paths $\wt{\phi}$. Similarly, one can define
$$\ell : \wt{\Ham} (X,\omega) \times HF(L) \rightarrow \R . $$

{\hamil} and {\lag} Floer homologies are related by the closed-open and open-closed maps
\begin{equation}
\begin{gathered}
\mathcal{CO}^0 : QH(X,\omega) \to HF(L),\\
\mathcal{OC}^0 :  HF(L) \to QH(X,\omega) ,
\end{gathered}
\end{equation}
 which are defined by counting certain holomorphic curves. The closed-open map $\mathcal{CO}^0$ is a ring homomorphism and the open-closed map $\mathcal{OC}^0$ defines a module action. As they are defined by counting certain holomorphic curves, which have positive $\omega$-energy, they have the following effect on {\specinv}s.
 
 \begin{prop}[{\cite{[BC09],[LZ18],[FOOO19]}}]\label{co oc spec inv}
 Let $H$ be any {\hamil}. 
 \begin{enumerate}
 \item For any $a \in QH(X,\omega)$, we have
 $$c(H,a) \geq \ell (H , \mathcal{CO}^0 (a) ) .$$
\item For any $\alpha \in HF(L)$, we have
 $$\ell (H , \alpha ) \geq  c(H,\mathcal{OC}^0 (\alpha)) .$$
 \end{enumerate}
 \end{prop}

\subsection{{\EP} {\qmor}s and (super)heaviness}\label{EP qmors and superheaviness}

Based on {\specinv}s, {\EP} built two theories, namely the theory of (Calabi) {\qmor}s and the theory of (super)heaviness, which we briefly review in this section. 

\textit{Quasimorphisms.} {\EP} constructed a special map on $\wt{\Ham}(X,\omega)$ called the {\qmor} for under some assumptions. Recall that a {\qmor} $\mu$ on a group $G$ is a map to the real line $\R$ that satisfies the following two properties:
\begin{enumerate}
    \item There exists a constant $C>0$ such that 
    $$|\mu(f \cdot g) -\mu(f)-\mu(g)|<C $$
    for any $f,g \in G$.
    
    \item For any $k \in \Z$ and $f \in G$, we have
    $$\mu(f^k)=k \cdot \mu(f) .$$
\end{enumerate} 

The following is {\EP}'s construction of {\qmor}s on $\wt{\Ham}(X,\omega)$.

\begin{theo}[{\cite{[EP03]}}]\label{EP qmor}
Suppose $QH(X,\omega ;\Lambda)$ has a field factor, i.e. 
$$ QH(X,\omega) = Q \oplus A $$
where $Q$ is a field and $A$ is some algebra. Decompose the unit $1_{X}$ of $QH(X,\omega )$ {\wrt} this split, i.e. 
$$1_{X}=e + a .$$
Then, the {\asympt} {\specinv} of $\wt{\phi}$ {\wrt} $e$ defines a {\qmor}, i.e.
\begin{equation}\label{EP qmor def}
    \begin{gathered}
        \ovl{c}_e=\zeta_{e}:\wt{\Ham}(X,\omega) \longrightarrow \R \\
         \ovl{c}_e ( \wt{\phi}) = \zeta_{e} ( \wt{\phi}) := \lim_{k \to +\infty} \frac{c (\wt{\phi} ^{ k},e )}{k} =\lim_{k \to +\infty} \frac{c (H^{\# k},e )}{k}
    \end{gathered}
\end{equation}
where $H$ is any mean-normalized {\hamil} such that the path $t \mapsto \phi_H ^t $ represents the class $\wt{\phi}$ in $\wt{\Ham}(X,\omega)$.
\end{theo}

\begin{remark}
    We refer to {\qmor}s on $\wt{\Ham}(X,\omega)$ defined as \eqref{EP qmor def} as {\EP} {\qmor}s. Notations $\ovl{c}_e,\zeta_{e}$ are both used to denote an {\EP} {\qmor}.
\end{remark}

\begin{remark}
By slight abuse of notation, we will also see $\zeta_{e}$ as a function on the set of time-independent {\hamil}s:
\begin{equation}
    \begin{gathered}
        \zeta_{e}:C^{\infty}(X) \longrightarrow \R \\
         \zeta_{e} ( H) := \lim_{k \to +\infty} \frac{c (H^{\# k},e )}{k}.
    \end{gathered}
\end{equation}
\end{remark}

\begin{remark}
  The {\lag} {\specinv}s do not appear in the result of {\EP}, but we define the \textit{asymptotic {\lag} {\specinv}s}, as we will use them later on in the proofs. 

\begin{equation}\label{asymp lag spec inv}
    \begin{gathered}
        \ovl{\ell}_L :\wt{\Ham}(X,\omega) \longrightarrow \R \\
         \ovl{\ell}_L := \lim_{k \to +\infty} \frac{ \ell ( \wt{\phi} ^{ k},1_L )}{k}
    \end{gathered}
\end{equation}  
\end{remark}

\textit{Superheaviness.} {\EP} introduced a notion of {\symp} rigidity for subsets in $(X,\omega)$ called (super)heaviness.

\begin{definition}[{\cite{[EP09]},\cite{[EP06]}}]\label{def of heavy}
Take an idempotent $e \in QH(X,\omega )$ and denote the {\asympt} {\specinv} {\wrt} $e$ by $\zeta_{e}$. A subset $S$ of $(X,\omega)$ is called
\begin{enumerate}
    \item $e$-heavy if for any time-independent {\hamil} $H:X \to \R$, we have
$$  \inf_{x\in S} H(x)  \leq \zeta_{e} ( H) , $$

\item $e$-{\suphv} if for any time-independent {\hamil} $H:X \to \R$, we have
$$ \zeta_{e} ( H) \leq \sup_{x\in S} H(x) . $$

\end{enumerate}
\end{definition} 

\begin{remark}
    Note that if a set $S$ is $e$-{\suphv}, then it is also $e$-heavy. 
\end{remark}

The following is an easy corollary of the definition of {\suphvness} which is useful.

\begin{prop}[{\cite{[EP09]}}]\label{suphv constant}
Assume the same condition on $QH(X,\omega )$ as in Theorem \ref{EP qmor}. Let $S$ be a subset of $X $ that is $e$-{\suphv}. For a time-independent {\hamil} $H:X \to \R$ whose restriction to $S$ is constant, i.e. $H|_{S}\equiv r,\ r\in \R$, we have
$$\zeta_e (H)=r .$$
In particular, two disjoint subsets of $(X,\omega)$ cannot be both $e$-{\suphv}.
\end{prop} 

\begin{proof}
    The first part is an immediate consequence of the definition of (super)heaviness. As for the second part, suppose we have two disjoint sets $A,B$ in $(X,\omega)$ that are both $e$-{\suphv}. Consider a {\hamil} $H$ that is
    $$H|_A=0,\ H|_B=1.$$
    Then, by {\suphvness}, we have
    $$1= \inf_{x\in B} H(x) \leq \zeta_{e}(H) \leq \sup_{x \in A} H(x) =0 ,$$
    which is a contradiction. 
\end{proof}

We end this section by giving a criterion for heaviness, proved by {\FOOO} (there are earlier results with less generality, c.f. \cite{Alb05}) using the closed-open map
$$\mathcal{CO}^0 : QH (X,\omega) \to HF  (L)  . $$

\begin{theo}[{\cite[Theorem 1.6]{[FOOO19]}}]\label{CO map heavy}

Assume $HF ( L ) \neq 0$. If 
$$\mathcal{CO}^0 (e)\neq 0$$
for an idempotent $e\in QH (X,\omega)$, then $L$ is $e$-heavy.
\end{theo}

\begin{remark}
When $\zeta_e$ is {\homo}, e.g. when $e$ is a unit of a field factor of $QH (X,\omega)$ and $\zeta_e$ is an {\EP} {\qmor}, then heaviness and {\suphvness} are equivalent so Theorem \ref{CO map heavy} will be good enough to obtain the {\suphvness} of $L$.
\end{remark}

\subsection{Donaldson divisors and Biran decomposition}\label{Donaldson divisors and Biran decomposition}

In this section, we briefly review the notion of Donaldson divisors in the sense of \cite{[BK13]} and explain an associated decomposition result due to Biran \cite{[Bir01],[Bir06]}. We will also see a way to construct a {\lag} {\smfd} in a closed {\symp} {\mfd} from a {\lag} {\smfd} in its Donaldson divisor, c.f. Definition \ref{donaldson}.

We first review the construction of {\symp} disk bundles.

Let $(\Sigma , \sigma)$ be an integral {\symp} {\mfd}. Take a complex line bundle 
$$\mathcal{L} \longrightarrow \Sigma$$
such that 
$$c_1(\mathcal{L} ) = [\sigma] .$$
Fix a hermitian metric $|\cdot|$ on $\mathcal{L}$ and a hermitian connection $\nabla$ such that the curvature form satisfies
$$R^{\nabla} = \frac{i}{2\pi} \sigma . $$
These datum determine the global angular form $\alpha^{\nabla}$, which satisfies
\begin{equation}
\begin{gathered}
    \alpha^{\nabla}|_{\Hor^{\nabla}} = 0,\\
    \alpha^{\nabla}_{(u)} (u) =0, \forall u \in \mathcal{L},\\
    \alpha^{\nabla}_{( u)} (i u) = \frac{1}{2\pi}, \forall u \in \mathcal{L},
\end{gathered}
\end{equation}
where $\Hor^{\nabla}$ denotes the horizontal distribution for $\nabla$. The global angular form $\alpha^{\nabla}$ satisfies
$$d \alpha^{\nabla} = - \pi^{\ast} \sigma .$$

The following form, which will be called the canonical {\symp} form on $\mathcal{L}$ induced by $(\Sigma , \sigma)$, defines a {\symp} structure on $\mathcal{L}$:
\begin{equation}
    \begin{gathered}
        \omega_{\can}:= -d ( e^{-r^2} \alpha^{\nabla}) = e^{-r^2} \pi^{\ast} \sigma + 2r e^{-r^2} dr \wedge \alpha^{\nabla} .
    \end{gathered}
\end{equation}
The radius $r$ part of the line bundle will be denoted as follows:
$$ \mathcal{L}_{\leq r}:= \{u\in \mathcal{L} : |u| \leq r  \} .$$

We now define Donaldson divisors.

\begin{definition}[{\cite[Section 2.3]{[BK13]}}]\label{donaldson}

Let $(X,\omega)$ be a closed integral {\symp} {\mfd}. A smooth hypersurface $\Sigma$ is called a Donaldson divisor if it satisfies the following properties:
\begin{enumerate}
    \item The homology class $\Sigma \in H_{2n-2} (X ;\Z)$ is Poincar\'e dual to $k [\omega] \in H^2 (X ; \Z)$ for some $k \in \N$.
    
    \item There exists a tubular neighborhood $U$ of $\Sigma$ in $X$ such that its closure is symplectomorphic to a standard {\symp} disk bundle $(\mathcal{L}_{\delta} , \frac{1}{k}\omega_{\can})$ over $(\Sigma , k \omega_{\Sigma})$ for some $\delta>0$, where $\omega_{\Sigma}$ denotes the restriction of $\omega$ to $\Sigma$.
    
    \item The complement of $\ovl{U}$ in $X$, i.e. $X \backslash \ovl{U} $ is a Weinstein domain.
\end{enumerate}

\end{definition}

Note that the integer $k$ that appear in the second column is called the \textit{degree} of the Donaldson divisor. We introduce two important classes of facts that give important examples of Donaldson divisors.

\begin{example}
\begin{enumerate}
    \item Let $X$ be a smooth projective variety and let $\Sigma$ be a smooth ample divisor. Denote the K\"ahler form $\omega$, which represents $c_1 (\mathcal{O}_{X} (\Sigma))$. Biran showed in \cite{[Bir01]} that $\Sigma$ is a Donaldson divisor of $(X,\omega)$.
    
    \item For any integral {\symp} {\mfd} $(X,\omega)$ and a sufficiently large $k \in \N$, there exists a Donaldson divisor $\Sigma$ such that $PD([\Sigma]) = k [\omega]$ , c.f. \cite{[Don96],[Gir02]}
\end{enumerate}
\end{example}
In order to emphasise that $\mathcal{L}$ is a fibration over $\Sigma$, we will denote
$$\mathcal{L}:= D\Sigma$$
from now on. Biran, in \cite{[Bir01]}, proved the following decomposition associated to a Donaldson divisor.

\begin{theo}\label{biran decomp}
Let $(X,\omega)$ be a closed integral {\symp} {\mfd} and let $\Sigma$ be a Donaldson divisor of degree $k$. Denote the canonical {\symp} disk bundle associated to $(\Sigma, k \omega_{\Sigma})$ by $(\mathcal{L} , \frac{1}{k}\omega_{\can})$. There exists a {\symp} embedding
$$ F : (D\Sigma , \frac{1}{k}\omega_{\can}) \hookrightarrow (X , \omega) $$
such that 
\begin{enumerate}
    \item $F(x, 0 ) =x$ for all $x\in \Sigma$, where $(x, 0 ) \in D\Sigma$ corresponds to a point in the zero section of $\mathcal{L}$.
    
    \item The set $\Delta:=X \backslash F(D\Sigma)$ is a isotropic CW-complex with respect to $\omega$.
    
    \item $(X \backslash F(D\Sigma_r) , \omega)$ is a Weinstein domain for all $r>0$.
    
    \item If the Weinstein manifold $X  \backslash \Sigma$ is subcritical, then $\Delta$ does not contain any {\lag} cell, thus $\dim (\Delta)< n = \frac{1}{2}\dim (X)$.
    \end{enumerate}

\end{theo}

We now explain the {\lag} circle bundle construction. Let $L$ be a {\lag} {\smfd} in $\Sigma$. Consider the radius $r>0$ circle bundle associated to the line bundle $ \mathcal{L} \to \Sigma$:
$$ \mathcal{L}_{ |u| = r}:= \{u\in \mathcal{L} : |u| = r  \} .$$
The set
$$ \wt{L}_r := \pi_{ |u| = r} ^{-1} (L),\ \pi_{ |u| = r}: \mathcal{L}_{ |u| = r} \to \Sigma $$
defines a {\lag} {\smfd} in $\mathcal{L}_{ }$, which is a circle bundle over $L$. Note that $\pi_{r}$ denotes the restricted projection $\mathcal{L}_{|u| = r} \to \Sigma$. Via the {\symp} embedding $F : (\mathcal{L} , \frac{1}{k}\omega_{\can}) \hookrightarrow (X , \omega) $, we can see $\wt{L}_r$ as a {\lag} {\smfd} in $X \backslash \Sigma $ or $X$, which we will always do in the sequel without mentioning.

When $L$ is a monotone in $\Sigma$, then there is a distinguished radius $r_0>0$ for which the lifted {\lag} {\smfd} $\wt{L}$ becomes also monotone in $X$ and according to \cite[{\propo} 6.4.1]{[BC09]}, it satisfies

\begin{equation}
    r_0 ^2 = \frac{2\kappa_L}{2\kappa_L+1}
\end{equation}
where $\kappa_L$ is the monotonicity constant for $L$ in $\Sigma$, i.e. $\omega_{\Sigma}|_{\pi_2(\Sigma,L)} = \kappa_L \cdot \mu_L|_{\pi_2(\Sigma,L)}$. We sometimes call the radius $r_0$ the monotone radius as well.

In the following, the lifted {\lag} {\smfd} $\wt{L}:=\wt{L}_{r_0}$ will always be this distinguished monotone {\lag} {\smfd} in $X$.

\begin{remark}
Given a Donaldson divisor $\Sigma$ of degree $k$ in $(X,\omega)$, we have $PD([\Sigma])= k[\omega]$. By rescaling the {\symp} form, we can assume that the divisor satisfies $PD([\Sigma])= [\omega]$ without loss of generality. We will always do this rescaling beforehand so that the various formulae related to {\symp} disk bundles and the Biran decomposition become simpler.
\end{remark}

\subsection{Borman's reduction of {\qmor}s}\label{Borman's reduction result}

In this section, we briefly review Borman's method from \cite{[Bor12]} to construct {\qmor}s on $\wt{\Ham}(\Sigma)$ from {\qmor}s on $\wt{\Ham}(X)$ where $X$ is a monotone {\symp} {\mfd} and $\Sigma$ is a Donaldson divisor therein.

We can see Biran decomposition as a special case of symplectization of a contact manifold in the following way: Let $(Y, \alpha)$ be a compact contact manifold. Now, assume that the Reeb flow defines a free $S^1$-action on $Y$. Then the quotient $\Sigma:=Y / S^1$ is a {\symp} {\mfd} and $Y \to \Sigma$ defines a circle bundle. Now, the symplectization of $(Y, \alpha)$, namely $(Y \times \R , d(e^{r}\alpha))$, is precisely the standard disk bundle over $\Sigma$. Now, given a {\hamil} $H$ on $\Sigma$, one can lift it to a {\hamil} on $X$ in the following way: take a smooth function 
 $$h:[0,1) \to \R$$
 such that $h(0)=1$ and it vanishes near $r=1$, and is a decreasing function. Consider the function
 \begin{equation}\label{Borman Ham lift}
     \begin{gathered}
         \R / \Z \times D\Sigma \to \R \\
         (t,x) \mapsto h(||x||) H\circ \pi (x)
     \end{gathered}
 \end{equation}
 where $||x||=h(x,x)^{1/2}$ is the radial coordinate of $x \in D\Sigma$ ({\wrt} to the hermitian metric $h$) and $\pi: D\Sigma \to \Sigma$ is the projection. This function can be regarded as a {\hamil} on $X$ via the symplectomorphism 
 $$F: D\Sigma  \xrightarrow[]{\simeq }  X \backslash \Delta$$
 and the trivial extension over $\Delta$. 
 Denote the set of paths of {\hamil} {\diffeo}s on a {\symp} {\mfd} $M$ by $\mathcal{P}\Ham(M)$:
 $$\mathcal{P}\Ham(M) :=\{ \gamma: [0,1] \to \Ham(M) : \gamma (0)=\id  \}.$$
 The extension of {\hamil}s on $\Sigma$ to $X$ described above defines a map 
 $$\Theta:  \mathcal{P}\Ham (\Sigma) \to \mathcal{P}\Ham (X), $$
 as elements of $\mathcal{P}\Ham (\Sigma)$ (resp. $\mathcal{P}\Ham (X)$) can be identified with  mean-normalized {\hamil}s on $\Sigma$ (resp. $X$). Borman proved the following.
 
 \begin{theo}[{\cite{[Bor12]}}]\label{Borman's reduction}
 Let $(X,\Sigma)$ be a pair of a closed {\symp} {\mfd} and a Donaldson divisor. Assume that the skeleton $\Delta$ is small (see Definition \ref{smallness def}). For any {\homo} {\qmor}
 $$\mu : \wt{\Ham}  (X) \to \R,$$
 the map 
 \begin{equation}\label{Borman reduction def}
     \begin{gathered}
         \wt{\Ham}  (\Sigma) \to \R \\
         \Theta^\ast \mu (\phi):= \mu \circ p_{X} \circ \Theta \circ p_{\Sigma} ^{-1} (\phi)  
     \end{gathered}
 \end{equation}
 is well-defined and defines a {\homo} {\qmor} on $\wt{\Ham}  (\Sigma)$, where 
  \begin{equation}
     \begin{gathered}
          p_{X}: \mathcal{P}\Ham(X) \to \wt{\Ham}  (X)=\mathcal{P}\Ham(X) / \text{homotopy equiv.}\\
          \gamma \mapsto [\gamma]
     \end{gathered}
 \end{equation}
 is the homotopy projection. The map $p_{\Sigma}$ is defined analogously.
 \end{theo}

 The following diagram summarises the reduction procedure:
 
\begin{equation}
\begin{tikzcd}
\mathcal{P}\Ham(\Sigma) \arrow[r,"\Theta"] \arrow[d,"p_{\Sigma}"] & \mathcal{P}\Ham(X)\arrow[d,"p_{X}"] & \ \\
\wt{\Ham}  (\Sigma) & \wt{\Ham}  (X) \arrow[r,"\mu"] & \R .
\end{tikzcd}
\end{equation}

 Note that the proof of Theorem \ref{Borman's reduction} uses only properties of {\qmor}s and does not appeal to Floer theory, thus it applies not only to {\EP} {\qmor}s and to any closed {\symp} {\mfd} without the monotonicity condition.

We precise the \textit{smallness} condition.

\begin{definition}\label{smallness def}
Let $(X,\Sigma)$ be a pair of a closed {\symp} {\mfd} and a Donaldson divisor and 
$$\zeta: \wt{\Ham} (X) \to \R$$
a {\qmor}. The skeleton $\Delta$ is called small {\wrt} $\zeta$ if there exists a small neighborhood $U$ of $\Delta$ such that $\zeta$ restricts to the Calabi homomorphism on $U$, i.e. for any {\hamil} $H$ supported in $U$, we have
$$ \zeta(H) = Cal (H):= \int_{t=0} ^1 \left( \int_{X} H(t,x) \omega^n \right) dt .$$
\end{definition}
 
In practice, the following properties are useful to decide that the skeleton is small. 

\begin{prop}\label{criterion for smallness}
Let $(X,\Sigma)$ be a pair of a closed {\symp} {\mfd} and a Donaldson divisor and 
$$\ovl{c}_e: \wt{\Ham} (X) \to \R$$
an {\EP} {\qmor}. If either one of the following conditions holds, then the skeleton $\Delta$ is small {\wrt} $\ovl{c}_e$:
\begin{enumerate}
    \item The complement of the skeleton $X\backslash \Delta$ is $\ovl{c}_e$-{\suphv}.
    
    \item The skeleton $\Delta$ has $\ovl{c}_e$-measure zero:
    $$\tau_{\ovl{c}_e} (\Delta) = 0 .$$
\end{enumerate}
\end{prop}

For the proof of {\propo} \ref{criterion for smallness}, see \cite{[EP09]}.

\section{Proofs}\label{proofs}

\subsection{Statement of the main result}

We start by restating Theorem \ref{main thm intro} with the notions from the premilinary section \ref{prelim}.

\begin{theox}\label{main thm}
Let $(X,\Sigma)$ be a pair of a closed monotone {\symp} {\mfd} and a Donaldson divisor therein. Assume that the following conditions are satisfied:
\begin{itemize}
    \item There exists a monotone {\lag} torus $L$ in $\Sigma$ whose {\suppot} $W_L$ has a non-degenerate critical point.
    \item The {\suppot} $W_{\wt{L}}$ of the lifted monotone {\lag} torus $\wt{L}$ in $X$ also has a  non-degenerate critical point.
\end{itemize}
Then, there exist 
\begin{enumerate}
    \item an {\idem} $e_{X} \in QH(X)$ that is a unit of a field factor for which the skeleton $\Delta$ is small {\wrt} $\ovl{ c}_{ e_{X} } $, and
    
    \item  an {\idem} $e_{\Sigma} \in QH(\Sigma)$ that is a unit of a field factor and satisfies
    $$   \Theta^\ast \ovl{ c}_{ e_{X} } =  \frac{1}{2\kappa_L+1} \ovl{c}_{e_{\Sigma}  }  $$
where $\Theta^\ast $ denotes Borman's pull-back of {\qmor}s \eqref{borman map} and $\kappa_L$ is the monotonicity constant for $L$ in $\Sigma$, i.e. $\omega_{\Sigma}|_{\pi_2(\Sigma,L)} = \kappa_L \cdot \mu_L|_{\pi_2(\Sigma,L)}$.
\end{enumerate}

\end{theox}

A sketch of the main ideas of the proof is given in Section \ref{Strategy}.

\subsection{{\BK}'s quantum Gysin sequence}\label{quantum Gysin}

In this section, we briefly review {\BK}'s quantum Gysin sequence \cite{[BK13]} but upgrading it to the $\C$-{\coeff} equipped with $\C^\ast$-local systems. Another important difference with \cite{[BK13]} is that, they work with general {\lag}s while we only work with {\lag} tori. Throughout the section, let $(X,\Sigma)$ be a pair of a closed monotone {\symp} {\mfd} and a Donaldson divisor therein. Consider a monotone {\lag} torus $L$ in $\Sigma$ equipped with a local system $\rho$ which is a non-degenerate critical point of the {\suppot} $W_L$ (provided that it exists), which we sometimes denote $(L,\rho)$ for clarity. Consider also a local system $\wt{\rho}$ for the lifted monotone {\lag} torus $\wt{L}$ in $X$ that is a non-degenerate critical point of the {\suppot} $W_{\wt{L}}$ (provided that it exists).

Now, consider a generic pearl data $\mathscr{D}:=(f,g, J)$ for $L$ in $\Sigma$ (see Section \ref{Floer and pearl} for definitions). We would like to define a nice {\acs} $\wt{J}$ on $X$ so that the projection $\pi: X \backslash \Delta \simeq D\Sigma \to \Sigma$ becomes $(\wt{J} , J)$-{\holo}, which will be crucial to relate {\holo} curves in $X$ and $\Sigma$ later on. First, we define an {\acs} $\wt{J}_{D\Sigma}$ on $X$ as follows:

For $v \in \Hor^{\nabla}$, 
\begin{equation}
    \begin{gathered}
        \wt{J}_{D\Sigma}(v):= (d \pi|_{\Hor^{\nabla}})^{-1} \circ J \circ d \pi (v)
    \end{gathered}
\end{equation}
and in the fiber component, define $\wt{J}$ to be multiplication by $i$, i.e. $\wt{J}_{D\Sigma}|_{\Vert^{\nabla}}=i$. Note that $\Hor^{\nabla}$ and $\Vert^{\nabla}$ denote the horizontal and vertical distributions, respectively, {\wrt} the connection $\nabla$ which we defined in Section \ref{Symplectic disk bundle}. Now, fix a small $\kappa >0$ and see $\wt{J}_{D\Sigma}$ an {\acs} in $F((D \Sigma)_{r_0 +\kappa} ) \subset X$ by using the embedding
$$F: D \Sigma \xrightarrow[]{\simeq} X \backslash \Delta $$
from the Biran decomposition (see Theorem \ref{biran decomp}). Extend this {\acs} to the complement $X \backslash F((D \Sigma)_{r_0 +\kappa} )$ in a generic way and denote the resulting {\acs} by $\wt{J}$.

{\BK} showed that by performing a neck-stretching to $\wt{J}$, we get a set of {\acs}s $\wt{\mathcal{J}}$, whose elements are called the admissible {\acs}s, that satisfies the following property:

\begin{prop}[{\cite[{\propo} 5.1]{[BK13]}}]\label{benefit SFT}
For any $\wt{J} \in \wt{\mathcal{J}}$, any $\wt{J}$-{\holo} curve $u:D \to X$ is contained in $F((D \Sigma)_{r_0 +\kappa} )$.
\end{prop}

We refer the readers to \cite[Section 5]{[BK13]} for details concerning this property. With an admissible {\acs} $\wt{J} \in \wt{\mathcal{J}}$, we define a pearl complex $C^\ast (\mathscr{D} ) $ for $\wt{L}$ with the pearl datum $\mathscr{D}:=(\wt{f} , \wt{X} , \wt{J})$  where $\wt{X}$ is a suitable almost gradient vector field that projects to $X$.

{\BK} show that even though an admissible {\acs} $\wt{J} \in \wt{\mathcal{J}}$ is not a generic {\acs} and therefore the general pearl theory is not directly applicable, the map 
$$\wt{d}: C^\ast (\mathscr{D} ) \to C^{\ast+1} (\mathscr{D} ),$$
which is defined in the same way as the pearl differential for generic {\acs}s, is well-defined with an admissible {\acs} $\wt{J} \in \wt{\mathcal{J}}$ and defines a cochain complex $(C^\ast (\mathscr{D} )  , \wt{d} )$ whose cohomology is isomorphic to the generic pearl homology $QH^\ast (\wt{L})$.

We define the relevant maps to construct the quantum Gysin sequence.

\begin{definition}\label{chain prop}
Define the maps $i,p$ as follows: 
\begin{equation}
    \begin{gathered}
        i: C(\mathscr{D} ) \to C(\wt{\mathscr{D} }) \\
        x \mapsto x' 
    \end{gathered}
\end{equation}
and
\begin{equation}
    \begin{gathered}
        p: C(\wt{\mathscr{D} }) \to C(\mathscr{D} ) \\
        x' \mapsto 0,\\
        x'' \mapsto x .
    \end{gathered}
\end{equation}
\end{definition}

We will prove that the maps $i,p$ are chain maps. This statement was proven in \cite[Section 7.3]{[BK13]} but without local systems. 

\begin{prop}\label{chain prop}
The maps $i,p$ are both chain maps.
\end{prop}

Once {\propo} \ref{chain prop} is proven, as the sequence
\begin{equation}
\begin{tikzcd}
0 \arrow[r,""] &  C(\mathscr{D} ) \arrow[r,"i_{}"] & C(\wt{\mathscr{D} })  \arrow[r,"p"] &  C(\mathscr{D} )\arrow[r,""] & 0
\end{tikzcd}
\end{equation}
is a short exact sequence, we get the following long exact sequence which we call the quantum Gysin sequence:

\begin{corol}[Quantum Gysin sequence] 
We have the following long exact sequence:
\begin{equation}\label{quantum Gysin sequence}
\begin{tikzcd}
\arrow[r,"\delta_{}"] & QH^{\ast} (L,\rho) \arrow[r,"i_{}"] & QH^{\ast} (\wt{L},\wt{\rho}) \arrow[r,"p"] & QH^{\ast-1} (L,\rho) \arrow[r,"\delta_{}"] &.
\end{tikzcd}
\end{equation}

\end{corol}

\begin{proof}[Proof of {\propo} \ref{chain prop}]

We only prove it for $i$ as the argument is similar for $p$. We will prove 
\begin{equation}
    \begin{gathered}
        d_{\wt{\rho}} \circ i = i \circ d_{\rho} .
    \end{gathered}
\end{equation}
We will use the following useful description of the pearl differential for monotone {\lag}s that admit a perfect Morse function due to {\BC} \cite{[BC12]}. Recall that we work with (\lag) tori so this applies to our case. We start by explaining the setup. Let $f:L \to \R$ be a perfect Morse function where $L$ is a $k=n-1$-dimensional {\lag} torus. Denote its critical points of index $1$ by $\{x_1, x_2 ,\cdots , x_{k=n-1}\}$. Each critical point in $\{x_1, x_2 ,\cdots , x_{k=n-1}\}$ represents a class $\gamma_j \in H_1(L):=H_1(L;\Z)/ \text{Torsion}$ and moreover, $\{\gamma_j\}_{1 \leq j \leq k}$ provides a basis of $H_1(L)$. 

Denote the {\suppot} for $L$ by $W_L$
\begin{equation}
    \begin{gathered}
        W_L : \hom (H_1(L) ,\C^{\ast} ) \to \C \\
        W_L (\rho) := \sum_{\mu(A)=2        , A\in H_2 (\Sigma,L)} \rho(\partial A) \cdot  \# \mathcal{M}(J;A) 
    \end{gathered}
\end{equation}
and by using the basis of $H_1(L)$ given by $\{\gamma_j\}_{1 \leq j \leq k}$, we can express $W_L$ in the following way:
$$W_L : \C^{k} \to \C $$
\begin{equation}\label{suppot expression}
    \begin{aligned}
         W_L(z) & =  \sum_{\mu(A)=2,  A\in H_2 (\Sigma,L)}   z^{\partial A}   \#  \mathcal{M}(J;A) \\
          & = \sum_{\mu(A)=2, A\in H_2 (\Sigma,L) }  z_{1} ^{ (\partial A)_{1}}  z_{2} ^{ (\partial A)_{2}} \cdots  z_{k} ^{ (\partial A)_{k}}   \#  \mathcal{M}(J;A) 
    \end{aligned}
\end{equation}
where $z_j$ is a variable representing the loop $\gamma_j$ and $(\partial A)_j $ are integers determined by 
\begin{equation}
    \begin{gathered}
        \partial A = \sum_{1 \leq j \leq k=n-1 } (\partial A)_j  \cdot \gamma_j \in H_1(L) .
    \end{gathered}
\end{equation}

{\BC} proved that the pearl differential of $\{x_1, x_2 ,\cdots , x_{k=n-1}\}$ is expressed in a rather simple way using the {\suppot}.

\begin{prop}[{\cite[Proposition 3.3.1]{[BC12]}}]\label{d and suppot}
Take any index $1$ critical point $x_j $ of $f: L \to \R$. We have 
\begin{equation}
    \begin{gathered}
        d_{\rho} (x_j)= z_j \frac{\partial W_L}{\partial z_j} x_{\min} t 
    \end{gathered}
\end{equation}
where $x_{\min}$ denotes the minimizer of the function $f$.
\end{prop}

\begin{remark}

Strictly speaking, {\BC} deal with index $k-1$ critical points but the idea, which we will explain shortly, is essentially the same as they work with homology while we are considering its cohomological counterpart. 

\end{remark}

The idea behind {\propo} \ref{d and suppot} is that, for degree reason and that $f$ is a perfect Morse function, we have
\begin{equation}
    \begin{gathered}
        d_{\rho} (x_j)= \sum_{\mu(A)=2 } \rho(\partial A) \# \mathscr{P} (x_{\min} , x_j ;A ) x_{\min} t,
    \end{gathered}
\end{equation}
and because $x_{\min}$ is a minimum of the Morse function $f$ and $x_j$ presents the class $\gamma_j$, $\#  \mathscr{P} (x_{\min} , x_j ;A )$ counts Maslov $2$-disks of class $A$ whose boundary $\partial A$ has $\gamma_j$ components. Thus

\begin{equation}
    \begin{gathered}
       \#  \mathscr{P} (x_{\min} , x_j ;A ) = (\partial A)_j \cdot \# \mathcal{M}(J;A). 
    \end{gathered}
\end{equation}
Thus, by using the expression \eqref{suppot expression}, we have 

\begin{equation}\label{LHS}
    \begin{aligned}
        d_{\rho} (x_j)& = \sum_{\mu(A)=2 } \rho(\partial A) \# \mathscr{P} (x_{\min} , x_j ;A ) x_{\min} t \\
        &= \sum_{\mu(A)=2 }  z_{1} ^{ (\partial A)_{1}}  z_{2} ^{(\partial A)_{2}} \cdots  z_{k} ^{ (\partial A)_{k}} \cdot (A)_j \cdot \# \mathcal{M}(J;A)  x_{\min} t \\
        &= z_j \frac{\partial W_L}{\partial z_j} x_{\min} t . 
    \end{aligned}
\end{equation}

Now, we turn our focus to $(X, \wt{L})$ and study the index $1$ critical points of $\wt{f}: \wt{L} \to \R$ of the type
$x'$, which are precisely the lifts of the index $1$ critical points of $f: L \to \R$, namely $\{x_1 ' , x_2 ' ,\cdots , x_{k=n-1} '\}$. Note that there is another index $1$ critical point of $\wt{f}: \wt{L} \to \R$ but of the type
$x ''$, which is a lift of the index $0$ critical point of $f: L \to \R$, i.e. the minimum, but we are not interested in this. We will prove that an analogous result of Proposition \ref{d and suppot} continues to hold for $\{x_1 ' , x_2 ' ,\cdots , x_{k=n-1} '\}$, which does not hold for the other index $1$ critical point $x ''$.

\begin{claim}\label{d and suppot lift}
Take any index $1$ critical point $x_j '$. We have
\begin{equation}
    \begin{gathered}
        d_{\wt{\rho}} (x_j ')= z_j \frac{\partial W_{\wt{L}}}{\partial z_j} x_{\min} ' t .
    \end{gathered}
\end{equation}
\end{claim}

We will prove this claim. Notice that each of the critical points $\{x_1 ' , x_2 ' ,\cdots , x_{k=n-1} '\}$, just as $\{x_1, x_2 ,\cdots , x_{k=n-1}\}$, represent a class in $H_1 (\wt{L})$, which we will denote by $\wt{\gamma}_{j}$. Denote the class in $H_1 (\wt{L})$ represented by the fiber circle by $\wt{\gamma}_{n}$. Now, $\{\wt{\gamma}_{j}\}_{1\leq j \leq n}$ provides a basis for $H_1 (\wt{L})$. Now, the {\suppot} for $\wt{L}$ {\wrt} to this basis will be

$$W_{\wt{L}}  : \C^{n} \to \C $$
\begin{equation}\label{lift suppot}
    \begin{aligned}
         W_{\wt{L}}(z)& =  \sum_{\mu(A)=2,  A\in H_2 (X, \wt{L})}   z^{\partial A}   \cdot \#  \mathcal{M}(J;A) \\
         & = \sum_{\mu(A)=2, A\in H_2 (X, \wt{L}) }  z_{1} ^{ (\partial A)_{1}}  z_{2} ^{ (\partial A)_{2}} \cdots  z_{n} ^{ (\partial A)_{n}}   \#  \cdot \mathcal{M}(J;A) 
    \end{aligned}
\end{equation}
where complex variables $z_j$ correspond to the class $\wt{\gamma}_{j}$ and integers $(\partial A)_j $ are determined by 
\begin{equation}
    \begin{gathered}
        \partial A = \sum_{1 \leq j \leq n } (\partial A)_j  \cdot \wt{\gamma}_{j} \in H_1(\wt{L}) .
    \end{gathered}
\end{equation}

Once again, for degree reason, we have
\begin{equation}
    \begin{gathered}
       d_{\wt{\rho}}   (x_j ' )= \sum_{i(\wt{x})=2} \# \mathscr{P} (\wt{x} , x_j ' ;0 )\wt{x}  +\sum_{\mu(A)=2 } \rho(\partial A) \# \mathscr{P} (x_{\min} , x_j ' ;A ) x_{\min} t.
    \end{gathered}
\end{equation}
The difference between here and its counterpart for $L$ is that, as $\wt{f}$ is no longer a perfect Morse function, we need to take into account the first term. The critical points satisfying $i(\wt{x})=2$ are either of the type $x'$ with $i(x)=2$ or of the type $x''$ with $i(x)=1$. Thus,

\begin{equation}
    \begin{gathered}
        \sum_{i(\wt{x})=2} \# \mathscr{P} (\wt{x} , x_j ' ;0 )\wt{x} =  \sum_{i(x)=1} \# \mathscr{P} (x'' , x_j ' ;0 ) x'' + \sum_{i(x)=2} \# \mathscr{P} (x' , x_j ' ;0 ) x' .
    \end{gathered}
\end{equation}
As it is explained in \cite[Section 7.2, equations 13, 14]{[BK13]}, one can show that 
\begin{equation}
    \begin{gathered}
      \mathscr{P} (x'' , x_j ';0 ) = \emptyset 
    \end{gathered}
\end{equation}
and
\begin{equation}\label{section 3 eq 1}
    \begin{gathered}
      \# \mathscr{P} (x' , x_j ';0 )  = \# \mathscr{P} (x , x_j  ;0 ) =0 
    \end{gathered}
\end{equation}
where the last equality \eqref{section 3 eq 1} uses that $f:L \to \R$ is perfect. Thus, we conclude that 
\begin{equation}
    \begin{gathered}
       d_{\wt{\rho}}   (x_j ' )= \sum_{\mu(A)=2 } \rho(\partial A) \# \mathscr{P} (x_{\min} , x_j ' ;A ) x_{\min} t.
    \end{gathered}
\end{equation}
Now, because $x_{\min} '$ is a minimum of the Morse function $\wt{f}$ and $x_j '$ presents the class $\wt{\gamma}_j$, $\#  \mathscr{P} (x_{\min} ' , x_j ' ;A )$ counts Maslov $2$-disks of class $A$ whose boundary $\partial A$ has $\wt{\gamma}_j$ components. Thus

\begin{equation}
    \begin{gathered}
       \#  \mathscr{P} (x_{\min}' , x_j ' ;A ) = (\partial A)_j \cdot \# \mathcal{M}(J;A). 
    \end{gathered}
\end{equation}
Thus, from \eqref{lift suppot}, we have 
\begin{equation}
    \begin{gathered}
        d_{\wt{\rho}} (x_j ')= z_j \frac{\partial W_{\wt{L}}}{\partial z_j} x_{\min} ' t .
    \end{gathered}
\end{equation}
which proves Claim \ref{d and suppot lift}.

Now, we are ready to prove Proposition \ref{chain prop}. As $H^\ast(\wt{L})=H^\ast(T^n)$ is generated by $H^1(\wt{L})=H^1(T^n)$ and $x_j$ represents a class $\gamma_j$ where $\{\gamma_j\}_{1 \leq j \leq k}$ forms a basis of $H_1(L)$, it is sufficient to prove
\begin{equation}
    \begin{gathered}
       i \circ d_{\rho} (x_j) = d_{\wt{\rho}} \circ i (x_j)
    \end{gathered}
\end{equation}
for any index $1$ critical point $x_j$ of $f$. From {\propo} \ref{d and suppot}, the left hand side is 
\begin{equation}
    \begin{gathered}
       i \circ d_{\rho} (x_j) = i( z_j \frac{\partial W_L}{\partial z_j} x_{\min} t ) =0
    \end{gathered}
\end{equation}
as we have taken a local system $\rho$ that is a critical point of the {\suppot} $W_L$. Next, Claim \ref{d and suppot lift}, the right hand side is 
\begin{equation}
    \begin{gathered}
      d_{\wt{\rho}} \circ i (x_j)= d_{\wt{\rho}} (x_j ') = z_j \frac{\partial W_{\wt{L}}}{\partial z_j} x_{\min} ' t =0
    \end{gathered}
\end{equation}
as we have taken a local system $\wt{\rho}$ that is a critical point of the {\suppot} $W_{\wt{L}}$. Thus, 
\begin{equation}\label{index 1 commute}
    \begin{gathered}
       i \circ d_{\rho} (x_j) = d_{\wt{\rho}} \circ i (x_j) =0
    \end{gathered}
\end{equation}
for any the index $1$ critical point $x_j$ of $f: L \to \R$. This completes the proof of {\propo} \ref{chain prop}.
\end{proof}

\subsection{The connecting map of the quantum Gysin sequence}\label{The connecting map}

In this section we prove that the connecting map $QH(L,\rho) \to QH(L,\rho)$ in quantum Gysin sequence for the {\lag} torus $L$ is a zero map.

\begin{prop}\label{connecting map zero}
The connecting map $\delta$ of the quantum Gysin sequence 
\begin{equation*}
\begin{tikzcd}
\arrow[r,"\delta_{}"] & QH^{\ast} (L,\rho) \arrow[r,"i_{}"] & QH^{\ast} (\wt{L},\wt{\rho}) \arrow[r,"p"] & QH^{\ast-1} (L,\rho) \arrow[r,"\delta_{}"] &.
\end{tikzcd}
\end{equation*}
is a zero map, i.e. 
$$\delta = 0.$$
\end{prop}

\begin{proof}[Proof of {\propo} \ref{connecting map zero}]

    The connecting map satisfies the following (\cite{[BK13]}): for classes $a,b \in QH(L,\rho)$, 
    \begin{equation}
    \begin{gathered}
       \delta(a \cdot b) = \delta(a ) \cdot b.
    \end{gathered}
\end{equation}

    Thus, in order to prove {\propo} \ref{connecting map zero}, it suffices to prove that 
    \begin{equation}
    \begin{gathered}
        {e}_{q}:=\delta(1_{L})  =0.
    \end{gathered}
\end{equation}
    We work in the same setup as in {\propo} \ref{chain prop} and use the same notation. By the definition of the connecting map and that the class $1_{L}$ is represented by the unique minimum $x_{\min}$ of the function $f$, we have the following:

    \begin{equation}
        \begin{aligned}
            {e}_{q} & = \delta (1_L) \\
            & = [i^{-1} \circ d_{\wt{\rho}} \circ p^{-1} (x_{\min})] \\
           & = [i^{-1} \circ d_{\wt{\rho}} (x_{\min} '')] \\
           & = [i^{-1} (\sum_{A , x'} \wt{\rho}(\partial A) \# \mathscr{P}( x ' ,  x_{\min} '' ;A) x ' t^{\ovl{\mu}(A)} 
           +  \sum_{A , x''} \wt{\rho}(\partial A) \# \mathscr{P}( x '' ,  x_{\min} '' ;A) x '' t^{\ovl{\mu}(A)}    )]  \\
        &    = [i^{-1} (\sum_{A , x'} \wt{\rho}(\partial A) \# \mathscr{P}( x ' ,  x_{\min} '' ;A) x  t^{\ovl{\mu}(A)}  ) ]
        \end{aligned}
    \end{equation}
    where the final equality uses the definition of the map $i$. For degree reasons, we have
    \begin{equation}
        \begin{aligned}
          \ & \sum_{A , x'} \wt{\rho}(\partial A) \# \mathscr{P}( x ' ,  x_{\min} '' ;A) x ' t^{\ovl{\mu}(A)} \\
           & = \sum_{ i(x')=2 }  \# \mathscr{P}( x ' ,  x_{\min} '' ;0) x ' + \sum_{\mu(A)=2 , i(x')=0} \wt{\rho}(\partial A) \# \mathscr{P}( x ' ,  x_{\min} '' ;A) x ' t^{} \\
           & = \sum_{ i(x')=2 }  \# \mathscr{P}( x ' ,  x_{\min} '' ;0) x ' + \left( \sum_{\mu(A)=2} \wt{\rho}(\partial A) \# \mathscr{P}( x_{\min} ' ,  x_{\min} '' ;A)\right) x_{\min} ' t^{}
        \end{aligned}
    \end{equation}
    where the final equality uses that $x_{\min} '$ is the unique index $0$ critical point of $\wt{f}$. Thus, 
    \begin{equation}
        \begin{aligned}
            {e}_{q} & = [i^{-1} (\sum_{ i(x')=2 }  \# \mathscr{P}( x ' ,  x_{\min} '' ;0) x ')] + [i^{-1} (\sum_{\mu(A)=2} \wt{\rho}(\partial A) \# \mathscr{P}( x_{\min} ' ,  x_{\min} '' ;A) x_{\min} ' t^{})] \\
            & = e(\wt{L}) + \left( \sum_{\mu(A)=2} \wt{\rho}(\partial A) \# \mathscr{P}( x_{\min} ' ,  x_{\min} '' ;A) \right)  1_L t
        \end{aligned}
    \end{equation}
  where we used the Morse theoretic interpretation of the classical Gysin sequence ($e(\wt{L})$ denotes the classical Gysin class for the $S^1$-bundle $\wt{L} \to L$) and $1_L= [x_{\min}] $. We know that $e(\wt{L})=0$ (as $\wt{L}$ and $L$ are tori), and 
  \begin{equation}
        \begin{gathered}
             \sum_{\mu(A)=2} \wt{\rho}(\partial A) \# \mathscr{P}( x_{\min} ' ,  x_{\min} '' ;A) = z_n \frac{\partial W_{\wt{L}}}{\partial z_n} 
        \end{gathered}
    \end{equation}
  as similar reasons to what we have discussed in the proof of {\propo} \ref{chain prop}: because $x_{\min} '$ is a minimum of the Morse function $\wt{f}$ and the Morse trajectory of $\wt{f}$ between $x_{\min} ''$ and $x_{\min} '$ presents the class $\wt{\gamma}_n$ of the fiber circle, $\#  \mathscr{P} (x_{\min} ' , x_{\min} '' ;A )$ counts Maslov $2$-disks of class $A$ whose boundary $\partial A$ has $\wt{\gamma}_n$ components. Thus

\begin{equation}
    \begin{gathered}
       \#  \mathscr{P} (x_{\min}' , x_{\min} '' ;A ) = (\partial A)_n \cdot \# \mathcal{M}(\wt{J};A), 
    \end{gathered}
\end{equation}
  which implies
  \begin{equation}
        \begin{aligned}
             \sum_{\mu(A)=2} \wt{\rho}(\partial A) \# \mathscr{P}( x_{\min} ' ,  x_{\min} '' ;A) & =\sum_{\mu(A)=2} \wt{\rho}(\partial A) (\partial A)_n \cdot \# \mathcal{M}(\wt{J};A) \\
            & =  z_n \frac{\partial W_{\wt{L}}}{\partial z_n} .
        \end{aligned}
    \end{equation}
    Thus, 
  \begin{equation}
        \begin{gathered}
            {e}_{q} = z_n \frac{\partial W_{\wt{L}}}{\partial z_n}  1_L t =0
        \end{gathered}
    \end{equation}
    as we chose a local system that is a critical point of the {\suppot} $W_{\wt{L}}$. This completes the proof of {\propo} \ref{connecting map zero}.

\end{proof}

\subsection{Floer--Gysin sequence}\label{Floer--Gysin sequence}

The aim of this section is to construct a self-Floer theoretic analogy, which we refer to the Floer--Gysin sequence, of the quantum Gysin sequence we have seen in Section \ref{quantum Gysin}.

First of all, notice that the lifted {\hamil}s considered by Borman \eqref{Borman Ham lift} are {\degen} ({\hamil} chords of $h\cdot \pi^\ast H$ from $\wt{L}$ to itself appear in $S^1$-families) and thus are not suitable for Floer theory. Thus, we slightly perturb them in the following way to make them non-{\degen}:
    \begin{equation}\label{Ham lift nondeg}
        \wt{H}:= \left( \eps h(r) \sum_{j} (\nu_j \circ \pi) f(\theta) \right) \wedge \left( h(r) (H\circ \pi) \right)
    \end{equation}
    which is a concatenation of 
    \begin{subequations}
    \begin{gather}
       h (r) (H\circ \pi)\\
        F(x):=\eps h(r) \sum_{j} (\nu_j \circ \pi) f(\theta)
    \end{gather}
    \end{subequations}
    where $f:S^1 \to \R$ is a Morse function that has precisely two critical points. Recall that $h:[0,1) \to \R$ is a decreasing function with $h(0)=1$ and is constantly zero away from a neighborhood of $r_0 \in (0,1)$. Note that the $S^1$-action is only defined on $X \backslash \Sigma \backslash \Delta \simeq D \Sigma \backslash \Sigma$, so the $S^1$-coordinate $\theta$ only makes sense on $X \backslash \Sigma \backslash \Delta$ but thanks to our choice of $h_1$, the function $h_1(r) f(\theta) $ makes sense on the entire $X$. Remember that even though $\wt{H}$ is defined on $D\Sigma \simeq X \backslash \Delta$, by our choice of $h:[0,1) \to \R$ (see Section \ref{Borman's reduction result}), it extends smoothly to $X$. We will study what the one-periodic orbits of $\wt{H}$ look like. 
    
    The {\symp} form $\omega_{D\Sigma}$ is defined by
    \begin{equation}
        \begin{aligned}
            \omega_{D\Sigma}:& = -d((1-r^2)\alpha)\\
         & = (1-r^2)\pi^\ast \omega_{\Sigma} + 2r dr \wedge \alpha . \\  
        \end{aligned}
    \end{equation}
    
    Note that we have the following equivalence:
    \begin{subequations}
    \begin{gather}
        \pi^\ast \omega_{\Sigma} = - d \alpha \\
         \omega_{\Sigma} = -\pi_\ast d \alpha
    \end{gather}
    \end{subequations}
    The {\hamil} vector fields of $h(||\cdot ||) \pi^\ast H$ and $F$ (around the critical points) are as follows:
    \begin{subequations}\label{ham vect field}
    \begin{gather}
        X_{h(||\cdot ||) \pi^\ast H} :=\frac{h'(r)}{2r}R_{\alpha} + \frac{h(r)}{1-r^2} \pi^\ast X_H\\
        X_{F} :=\frac{h'(r)}{2r}f(\theta)R_{\alpha} + f'(\theta)\frac{h(r)}{2r} \partial_r.
    \end{gather}
    \end{subequations}
    Thus, we get 
    \begin{equation}
        \phi_{\wt{H}} (\wt{L}) \cap \wt{L}= \bigcup_{j} \{x_j ' , x_j '' \}
    \end{equation}
    where 
    $$\pi(x_j ')=\pi(x_j '') = x_j,\ \phi_H (L) \cap L = \bigcup_j \{x_j\} . $$

Following \cite{[Bor12]}, we take 
$$h(r) = 1-r^2$$
for $r \in [0,1-\eps]$ for a sufficiently small $\eps>0$ so that we have
$$\pi_\ast X_{h(||\cdot ||) \pi^\ast H} = X_H .$$
 
Thus, the generators of $CF_{X} (\wt{L},\wt{H}) $ are the following capped {\hamil} chords:
\begin{equation}
        \begin{gathered}
           [ \phi_{\wt{H}} ^{t} (x_j ') , \wt{u} ],\\
           [ \phi_{\wt{H}} ^{t} (x_j '') , \wt{u} ]
        \end{gathered}
    \end{equation}
    where $\wt{u}$ is the capping of the {\hamil} chord $\phi_{\wt{H}} ^{t} (x_j '),\phi_{\wt{H}} ^{t} (x_j '')$ that satisfy
    \begin{equation}
        \begin{gathered}
          \mu ([ \phi_{\wt{H}} ^{t} (x_j ') , \wt{u} ]) = \mu ([ \phi_{H} ^{t} (x_j ) , u ]),\\
          \mu ([ \phi_{\wt{H}} ^{t} (x_j '') , \wt{u} ]) = \mu ([ \phi_{H} ^{t} (x_j ) , u ]) +1 .
        \end{gathered} 
    \end{equation} 
    We explain this a bit more in detail. As the chord $\phi_{\wt{H}} ^{t} (\wt{x_j} )$ is a concatenation of the chords $\phi_F ^t(x_j ')$ and $\phi_{h\cdot \pi^\ast H} ^{t} (x_j ')$, where the former is a constant chord, it is geometrically identical to the latter. Recall that  $\phi_{h\cdot \pi^\ast H} ^{t} (x_j ')$ is contained in $F(D\Sigma_{r_0})$. One can take a capping $\wt{u} $ of $\phi_{h\cdot \pi^\ast H} ^{t} (x_j ')$ contained in $F(D\Sigma_{r_0})$ that satisfies
    $$ \mu( [ \phi_{h\cdot \pi^\ast H} ^{t} (x_j ') , \wt{u} ] ) = \mu ([ \phi_{H} ^{t} (x_j ) , u ]) . $$
   Thus, 
  \begin{equation}
        \begin{aligned}
          \mu ([ \phi_{\wt{H}} ^{t} (x_j ') , \wt{u} ]) & = \mu ([ \phi_{H} ^{t} (x_j ) , u ])+ i_{Morse}(x_j ') \\
          & = \mu ([ \phi_{H} ^{t} (x_j ) , u ])+0,\\
          \mu ([ \phi_{\wt{H}} ^{t} (x_j '') , \wt{u} ]) & = \mu ([ \phi_{H} ^{t} (x_j ) , u ]) + i_{Morse}(x_j '') \\
          & = \mu ([ \phi_{H} ^{t} (x_j ) , u ])+1 .
        \end{aligned} 
    \end{equation}

    \begin{remark}
   We will estimate the actions of the lifted {\hamil} chords $[ \phi_{\wt{H}} ^{t} (\wt{x_j} ) , \wt{u} ]$ even though it will not be useful until Section \ref{Filtered Floer--Gysin sequence}. First of all, the action of $[\wt{z}, \wt{u}]:= [ \phi_{h\cdot \pi^\ast H} ^{t} (x_j ') , \wt{u} ] $ satisfies the following:
    
    \begin{equation}
    \begin{aligned}
         \mathscr{A}_{\wt{L} , h\cdot \pi^\ast H} ([\wt{z}, \wt{u}])& =\int_0 ^1 h\cdot \pi^\ast H (\wt{z}(t)) dt - \int_{D_+} \wt{u}^{\ast} \omega_{D\Sigma}\\
        &=\int_0 ^1 (h \cdot \pi ^{\ast} H_t ) (\wt{z}(t)) dt - \int_{D_+} \wt{u}^{\ast} ((1-r^2)\pi^\ast \omega_{\Sigma} + 2r dr \wedge \alpha) \\
        &= \int_0 ^1 (1-r_0 ^2)  H_t (z(t)) dt - \int_{D_+} ( (1-r_0 ^2) (\wt{u} \circ \pi )^\ast \omega_{\Sigma} +  \wt{u}^{\ast}(2r dr \wedge \alpha) ) \\
        &= (1-r_0 ^2)  \int_0 ^1   H_t (z(t)) dt - \int_{D_+}  u^\ast \omega_{\Sigma} - \int_{D_+} \wt{u}^{\ast}( 2r dr \wedge \alpha) \\
        &= (1-r_0 ^2) \mathscr{A}_{L,H} ([z,u]),
    \end{aligned}
    \end{equation}
    where $[z,u]:=[\phi_H ^t (x_j), u ] $. Thus, the actions of $ [ \phi_{h\cdot \pi^\ast H} ^{t} (x_j ') , \wt{u} ] $ and $ [ \phi_{h\cdot \pi^\ast H} ^{t} (x_j '') , \wt{u} ] $ satisfy
    
    \begin{equation}\label{action estimate}
    \begin{gathered}
         \mathscr{A}_{\wt{L} , \wt{H} } ([ \phi_{\wt{H}} ^{t} (x_j ') , \wt{u} ] )=
         (1-r_0 ^2) \mathscr{A}_{L,H} ([z,u]) + \eps h(r_0) f(\theta ') , \\
        \mathscr{A}_{\wt{L} , \wt{H} } ([ \phi_{\wt{H} } ^{t} (x_j '') , \wt{u} ] )=
         (1-r_0 ^2) \mathscr{A}_{L,H} ([z,u]) +\eps h(r_0) f(\theta'') 
    \end{gathered}
    \end{equation}
    where $\theta'$ and $\theta ''$ are points that satisfy the following:
    \begin{equation}\label{action estimate}
    \begin{gathered}
          \min_{\theta \in S^1} f(\theta) = f(\theta '),\\
           \max_{\theta \in S^1} f(\theta) = f(\theta '') .
    \end{gathered}
    \end{equation}
    \end{remark}

We are now ready to define the Floer--Gysin sequence. We will introduce maps $i_{Fl},p_{Fl}$, which are analogies of the maps $i,p$ in the pearl case; compare with Definition \ref{chain prop}.

\begin{definition}\label{def chain prop Floer}
We define $i_{Fl}$ as 
\begin{equation}
    \begin{aligned}
        i_{Fl} : CF^{\ast} (L,H)  & \to CF^{\ast} (\wt{L},\wt{H})  \\
        [ \phi_{H} ^{t} (x_j ) , u ] & \mapsto [ \phi_{\wt{H}} ^{t} (x_j ') , \wt{u} ]
    \end{aligned}
\end{equation}
where $\wt{u} $ is a disk chosen above, i.e. a disk that makes 
\begin{equation}
    \begin{gathered}
       \mu ([ \phi_{H} ^{t} (x_j ) , u ] ) = \mu ([ \phi_{\wt{H}} ^{t} (x_j ') , \wt{u} ]),
    \end{gathered}
\end{equation}
and $p_{Fl}$ as
\begin{equation}
    \begin{aligned}
         p_{Fl}: CF^{\ast} (\wt{L},\wt{H}) & \to CF^{\ast-1} (L,H)  \\
          [ \phi_{\wt{H}} ^{t} (x_j ' ) , \wt{u} ] & \mapsto 0  ,\\
          [ \phi_{\wt{H}} ^{t} (x_j '') , \wt{u} ] & \mapsto [ \phi_{H} ^{t} (x_j ) , u ]  .
    \end{aligned}
\end{equation}
where $u = \pi (\wt{u})$.
\end{definition}

By arguing as \cite{[BK13]}, one can check that $i_{Fl},p_{Fl}$ are chain maps, and moreover, by considering the lifted PSS map
\begin{equation}
    \begin{aligned}
        \wt{PSS}:  C(\mathscr{D}) & \to CF_X  (\wt{L} , \wt{H}) \\
        x' & \mapsto  \sum_{\mu(\mathcal{T})=0 } \# \mathscr{P}( \phi_{\wt{H}} ^t (y' ), x'  ; \mathcal{T} ) \phi_{\wt{H}} ^t (y' ) t^{\ovl{\mu} (\mathcal{T}) } ,\\
        x'' & \mapsto   \sum_{\mu(\mathcal{T})=0 } \# \mathscr{P}( \phi_{\wt{H}} ^t (y' ), x''  ; \mathcal{T} ) \phi_{\wt{H}} ^t (y' ) t^{\ovl{\mu} (\mathcal{T}) } \\
        & +  \sum_{\mu(\mathcal{T})=0 } \# \mathscr{P}( \phi_{\wt{H}} ^t (y'' ), x''  ; \mathcal{T} ) \phi_{\wt{H}} ^t (y'' ) t^{\ovl{\mu} (\mathcal{T}) } ,
    \end{aligned}
    \end{equation}
we can see that quantum and Floer Gysin sequences are compatible.

\begin{prop}\label{compatibility chain level}
The following diagram commutes:

\begin{equation}
\begin{tikzcd}
\arrow[r,"\delta"] & C( \mathscr{D}) \arrow[r,"i"] \arrow[d,"PSS"] & C(\wt{\mathscr{D}})  \arrow[r,"p"] \arrow[d,"\wt{PSS}"] & C( \mathscr{D}) \arrow[d,"PSS"] \arrow[r,"\delta"]  & \ \\
\arrow[r,"\delta_{Fl}"] & CF_{\Sigma}  (L,H) \arrow[r,"i_{Fl}"] & CF_{X}  (\wt{L},\wt{H}) \arrow[r,"p_{Fl}"] & CF_{\Sigma} (L,H) \arrow[r,"\delta_{Fl}"] &.
\end{tikzcd}
\end{equation}

\end{prop}

This immediately implies the following.

\begin{corol}\label{compatibility homology level}
The following diagram commutes:

\begin{equation}\label{Floer-Gysin seq 2}
\begin{tikzcd}
\arrow[r,"\delta"] & QH^{\ast}(L) \arrow[r,"i"] \arrow[d,"PSS"] & QH^{\ast}(\wt{L}) \arrow[r,"p"] \arrow[d,"\wt{PSS}"] & QH^{\ast -1}(L) \arrow[d,"PSS"] \arrow[r,"\delta"]  & \ \\
\arrow[r,"\delta_{Fl}"] & HF_{\Sigma}  (L,H) \arrow[r,"i_{Fl}"] & HF_{X}  (\wt{L},\wt{H}) \arrow[r,"p_{Fl}"] & HF_{\Sigma} (L,H) \arrow[r,"\delta_{Fl}"] &.
\end{tikzcd}
\end{equation}
\end{corol}

This compatibility of Floer--Gysin and quantum Gysin sequences ({\cor} \ref{compatibility homology level}) and $\delta=0$ ({\propo} \ref{connecting map zero}) imply the following:

\begin{corol}\label{connecting map Floer}
The connecting map of the Floer--Gysin sequence is a zero map, i.e. 
$$\delta_{Fl}=0 .$$
\end{corol}

\subsection{Filtered Floer--Gysin sequence}\label{Filtered Floer--Gysin sequence}

In this section, we study the change of filtration in the Floer--Gysin sequence \eqref{Floer-Gysin seq 2}. The main result of the section is the following:

\begin{prop}\label{filtered Floer--Gysin prop}
Let $H$ be any non-{\degen} {\hamil} on $\Sigma$ and define $\wt{H}$, which is a lifted non-{\degen} {\hamil} on $X$, defined by the equation \eqref{Ham lift nondeg}. The filtration change in the Floer--Gysin sequence for these {\hamil}s are as follows:

\begin{equation}\label{filtered Floer--Gysin seq}
\begin{tikzcd}
\arrow[r,"\delta_{Fl}"] & HF_{\Sigma} ^{\tau} (L,H) \arrow[r,"i_{Fl}"] & HF_{X} ^{h(r_0) \tau + \eps' } (\wt{L},\wt{H}) \arrow[r,"p_{Fl}"] & HF_{\Sigma} ^{\tau} (L,H) \arrow[r,"\delta_{Fl}"] & \
\end{tikzcd}
\end{equation}
where $\eps':= \eps \max_{x\in X} f(\theta) h(r) $.
\end{prop}

    \begin{proof}[Proof of {\propo} \ref{filtered Floer--Gysin prop}]
    
    From the definition of the maps $i_{Fl}$ and $p_{Fl}$, we need to estimate the change of action in 
    \begin{equation}
    \begin{gathered}
       i_{Fl}:  CF (L,H) \to CF (\wt{L} , \wt{H}) \\
       [\phi_H ^t (x_j) , u ] \mapsto [\phi_{\wt{H}} ^t (x_j ') , \wt{u} ]
    \end{gathered}
    \end{equation}
    and 
    \begin{equation}
    \begin{gathered}
        p_{Fl}:  CF (\wt{L} , \wt{H}) \to CF (L,H) \\
        [\phi_{\wt{H}} ^t (x_j '') , \wt{u} ] \mapsto [\phi_H ^t (x_j) , u ] .
    \end{gathered}
    \end{equation}

    The equation \eqref{action estimate} implies
    
    \begin{equation}
    \begin{gathered}
        CF_\Sigma ^{\tau} (L,H) \to CF_X ^{(1-r_0 ^2)\cdot \tau + \eps'}  (\wt{L} , \wt{H}) \\
        [\phi_H ^t (x_j) , u ] \mapsto [\phi_{\wt{H}} ^t (x_j ') , \wt{u} ]
    \end{gathered}
    \end{equation}
and 
    \begin{equation}
    \begin{gathered}
         CF_X ^{(1-r_0 ^2)\cdot \tau + \eps'}  (\wt{L} , \wt{H})  \to CF_\Sigma ^{\tau} (L,H)  \\
         [\phi_{\wt{H}} ^t (x_j '') , \wt{u} ] \mapsto [\phi_H ^t (x_j) , u ] .
    \end{gathered}
    \end{equation}
    as we have
    $$\eps h(r_0) \min f  <  \eps h(r_0) \max f < \eps' $$
    from our choice of $\eps$. This finishes the proof of {\propo} \ref{filtered Floer--Gysin prop}.
    \end{proof}

\subsection{Completing the proof}\label{Completing the proof}

We complete the proof of Theorem \ref{main thm intro} (a.k.a \ref{main thm}).

\begin{proof}[Proof of Theorem \ref{main thm intro} (a.k.a \ref{main thm})]
Notice that from the definition of the reduction explained in Section \ref{Borman reduction def}, what we need to prove is the following:
\begin{equation}\label{eq goal}
   \ovl{ c}_{e_X}  (h \pi^\ast H) = \frac{1}{2\kappa_L+1}  \ovl{c}_{e_\Sigma  } ( H)   
\end{equation}
for any {\hamil} $H$ on $\Sigma$. 

 First of all, we state the following property for {\lag} tori.
   
    \begin{prop}[{\cite[{\propo} 5.3]{[San21]}}]\label{Sanda}
    If $K$ is a {\lag} torus in $M$ that corresponds to a non-degenerate critical point of its {\suppot}, then the class
    \begin{equation}
        e_{M,K}:= \frac{1}{\langle p_{K}, p_{K} \rangle}_{\text{Muk}} \mathcal{OC}^0 (pt_{K}),
    \end{equation}
    where $pt_{K}$ is the point class in $HF(K)$, satisfies the following two properties:
    \begin{enumerate}
        \item It is an idempotent which is a unit of a field factor of $QH(M)$, i.e. $e_{L} \cdot QH(X)$ is a field. 
        \item It is mapped to $1_K$ by the closed-open map:
         \begin{equation}
     \begin{gathered}
       \mathcal{CO}^0 : QH(M) \to HF(K) \\
        \mathcal{CO}^0 (e_{M,K} ) = 1_K .
     \end{gathered}
 \end{equation}
    \end{enumerate}
    
    \end{prop}
    
    \begin{remark}
    The bracket $\langle -, - \rangle_{\text{Muk}}$ denotes the Mukai pairing, which is a canonical pairing for $HF(K)$ (more precisely, for the Hochschild homology of the Fukaya category $HH_\ast (\mathscr{F}(M))$), see \cite{[San21]}. Here, the only important thing is that $\frac{1}{\langle p_{K}, p_{K} \rangle}_{\text{Muk}}$ is an element of $\Lambda$ that only depends on the {\lag} $K$.
    \end{remark}
    
    As we are assuming that $L$ (resp. $\wt{L}$) is a monotone {\lag} torus equipped with a local system corresponding to a non-{\degen} critical point of its {\suppot}, by {\propo} \ref{Sanda}, there exists a unit of a field factor $e_{\Sigma}:=e_{\Sigma,L}$ (resp. $e_{X}:=e_{X,\wt{L}}$) of $QH(\Sigma)$ (resp. $QH(X)$). Thus, we have {\EP} {\qmor}s 
    $$\ovl{ c}_{ e_{X} } : \wt{\Ham} (X) \longrightarrow \R$$
    and 
    $$\ovl{ c}_{ e_{\Sigma} } : \wt{\Ham} (\Sigma) \longrightarrow \R . $$
    
   We will now prove the equation \eqref{eq goal}.

    Consider the diagram

\begin{equation}\label{main diagram in the proof}
\begin{tikzcd}
\arrow[r] & QH(L)  \arrow[r,"i"] \arrow[d,"PSS"]  & QH (\wt{L } ) \arrow[d,"\wt{PSS}"] \arrow[r,"p"] & QH(L)  \arrow[d,"PSS"]  \arrow[r] &  \ \\
  \arrow[r] & HF ^{\tau}(L, H)  \arrow[r,"i"] & HF ^{\frac{1}{2\kappa_L+1} \cdot \tau + \eps' } (\wt{L},  \wt{H}) \arrow[shift right, swap]{d}{\mathcal{OC}^{0}} \arrow[r,"p"] & HF ^{\tau}(L, H)  \arrow[d,"\mathcal{OC}^{0}"] \arrow[r] &  \ \\
      & QH (\Sigma) \arrow[u,"\mathcal{CO}^{0}"] & QH(X) \arrow[shift right, swap]{u}{\mathcal{CO}^{0}}  & QH( \Sigma)   & 
\end{tikzcd}
\end{equation}

From Corollary \ref{connecting map Floer}, it follows that
$$i:  QH (L) \to QH (\wt{L}) $$
is injective and 
    $$p: QH (\wt{L}) \to QH (L) $$
    is surjective. Thus, 
    $$i(1_L) = 1_{\wt{L}}$$
    and there exists a class $\alpha \in QH (\wt{L}) $ such that 
    $$p(\alpha)=1_{L}.$$

By focusing on the left side of the diagram \eqref{main diagram in the proof}, we get

\begin{equation}\label{diagram left}
\begin{tikzcd}
  \arrow[r] & HF ^{\tau}(L, H)  \arrow[r,"i"] & HF ^{\frac{1}{2\kappa_L+1} \cdot \tau +\eps'} (\wt{L},  \wt{H}) \arrow[shift right, swap]{d}{\mathcal{OC}^{0}} \arrow[r,] & \  \\
      & QH (\Sigma) \arrow[u,"\mathcal{CO}^{0}"] & QH(X)  &    .
\end{tikzcd}
\end{equation}

By using the diagram \eqref{diagram left} and basic properties of {\specinv}s, we get

\begin{equation}\label{diag left}
    \begin{aligned}
         \ell_{L} (H)& = \ell_{ \mathcal{CO}^{0}(e_{\Sigma}) } (H) \leq c_{e_{\Sigma}} (H) , \\
        \ell_{\wt{L}} (\wt{H})   
        & \leq \frac{1}{2\kappa_L+1}  \ell_{L}  (H)  +\eps' , \\
          \ell_{p_{\wt{L}}} (\wt{H})  
          & \leq \ell_{\wt{L}}(\wt{H}) ,  \\
          c_{\mathcal{OC}^0 (p_{\wt{L}}) } (\wt{H}) & \leq \ell_{p_{\wt{L}}} (\wt{H}) ,\\
        c_{e_X }(\wt{H}) = c_{\frac{1}{\langle p_{\wt{L}}, p_{\wt{L}} \rangle}_{\text{Muk}} \mathcal{OC}^0 (p_{\wt{L}}) } (\wt{H}) & = c_{ \mathcal{OC}^0 (p_{\wt{L}}) }(\wt{H}) - \nu ({\langle p_{\wt{L}}, p_{\wt{L}} \rangle}_{\text{Muk}} ^{-1}).
    \end{aligned}
\end{equation}

The chain of inequalities \eqref{diag left} imply 

\begin{equation}\label{eq 20}
      c_{e_X}  (\wt{H}) - \nu ({\langle p_{\wt{L}}, p_{\wt{L}} \rangle}_{\text{Muk}})   \leq \frac{1}{2\kappa_L+1}  c_{e_\Sigma} (H) +\eps'.
\end{equation}
From the triangle equality, we have

\begin{equation}\label{eq 21}
    \begin{gathered}
        c_{e_X}  (\wt{H})= c_{e_X}  (h \pi^\ast H \wedge F )  \geq  c_{e_X}  (h \pi^\ast H) - c_{e_X}  (\ovl{F}) 
        \geq c_{e_X}  (h \pi^\ast H) -  \eps'.
    \end{gathered}
\end{equation}
Thus, by combining the equations \eqref{eq 20} and \eqref{eq 21},
\begin{equation}
    \begin{gathered}
       \frac{1}{2\kappa_L+1}  c_{e_\Sigma} (H) +\eps'  \geq 
       c_{e_X}  (h \pi^\ast H) -  \eps' -  \nu ({\langle p_{\wt{L}}, p_{\wt{L}} \rangle}_{\text{Muk}})
    \end{gathered}
\end{equation}
and as we can take $\eps'$ arbitrarily small, we have 
\begin{equation}\label{x sigma}
    \begin{gathered}
        \frac{1}{2\kappa_L+1}  c_{e_\Sigma} (H) \geq c_{e_X}  (h \pi^\ast H) -  \nu ({\langle p_{\wt{L}}, p_{\wt{L}} \rangle}_{\text{Muk}}).
    \end{gathered}
\end{equation}

Next, by focusing on the right side of the diagram \eqref{main diagram in the proof}, we get

\begin{equation}\label{diagram right}
\begin{tikzcd}
  \arrow[r] & HF ^{\frac{1}{2\kappa_L+1} \cdot \tau +\eps' } (\wt{L},  \wt{H})  \arrow[r,"p"] & HF ^{\tau}(L, H)  \arrow[d,"\mathcal{OC}^{0}"] \arrow[r,""] &  \ \\
       & QH(X) \arrow{u}{\mathcal{CO}^{0}}  & QH( \Sigma)   & 
\end{tikzcd}
\end{equation}

By using the diagram \eqref{diagram right} and basic properties of {\specinv}s, we get

\begin{equation}\label{diag right}
    \begin{aligned}
          \ell_{\wt{L}} (\wt{H})& =\ell_{\mathcal{CO}^{0}(e_X)} (\wt{H}) \leq c_{e_X} (\wt{H}) , \\
            \ell_{\alpha}  (\wt{H})   & \leq \ell_{\wt{L}} (\wt{H})  + \nu_{\wt{L}} (\alpha)   \\
             \frac{1}{2\kappa_L+1}  \ell_{L} (H)  +\eps' & \leq \ell_{\alpha} (\wt{H}) , \\
          c_{\mathcal{OC}^0 (p_{L}) } (H) & \leq \ell_{p_{L}} (H) , \\
        c_{e_\Sigma } (H) = c_{\frac{1}{\langle p_{L}, p_{L} \rangle}_{\text{Muk}} \mathcal{OC}^0 (p_{L}) } (H) & = c_{ \mathcal{OC}^0 (p_{L}) } (H) - \nu ({\langle p_{L}, p_{L} \rangle}_{\text{Muk}} ^{-1})
    \end{aligned}
\end{equation}
where $ p(\alpha) = 1_{L}$. This chain of inequalities imply 
\begin{equation}\label{eq 22}
      \frac{1}{2\kappa_L+1}  ( c_{e_\Sigma  } (H) - \nu ({\langle p_{L}, p_{L} \rangle}_{\text{Muk}}) ) +\eps' \leq  c_{e_X} (\wt{H})+   \nu_{\wt{L}} (\alpha) .
\end{equation}
From the triangle equality, we have
\begin{equation}\label{eq 23}
    \begin{gathered}
        c_{e_X}  (\wt{H}) = c_{e_X}  (h \pi^\ast H \wedge F )  \leq  c_{e_X}  (h \pi^\ast H) + c_{e_X}  (F) 
        \leq c_{e_X}  (h \pi^\ast H) +  \eps'.
    \end{gathered}
\end{equation}
Thus, by combining the equations \eqref{eq 22} and \eqref{eq 23},
\begin{equation}
    \begin{gathered}
       \frac{1}{2\kappa_L+1}  ( c_{e_\Sigma  } (H) - \nu ({\langle p_{L}, p_{L} \rangle}_{\text{Muk}}) ) +\eps' \leq c_{e_X}  (h \pi^\ast H) +  \eps' +   \nu_{\wt{L}} (\alpha) 
    \end{gathered}
\end{equation}
and as we can take $\eps'$ arbitrarily small, we have 
\begin{equation}\label{sigma x}
       \frac{1}{2\kappa_L+1}  ( c_{e_\Sigma  } (H) + \nu ({\langle p_{L}, p_{L} \rangle}_{\text{Muk}}) )  \leq c_{e_X}  (h \pi^\ast H) +   \nu_{\wt{L}} (\alpha) .
\end{equation}

The equations \eqref{x sigma} and \eqref{sigma x} give 
\begin{equation}\label{final equation}
\begin{gathered}
   | c_{e_X } (h \pi^\ast H)  -  \frac{1}{2\kappa_L+1}  c_{e_\Sigma } (H)| \leq \text{Const}
\end{gathered}
\end{equation}
where
\begin{equation}
     \text{Const}:= \max \{|\nu ({\langle p_{L}, p_{L} \rangle}_{\text{Muk}}) ) - \nu_{\wt{L}} (\alpha)| , \nu ({\langle p_{\wt{L}}, p_{\wt{L}} \rangle}_{\text{Muk}}) \} .
\end{equation}
In particular, by homogenizing the equation \eqref{final equation}, we get
\begin{equation}\label{absolute spec inv descend}
   \ovl{ c}_{e_X}  (h \pi^\ast H) = \frac{1}{2\kappa_L+1}  \ovl{c}_{e_\Sigma  } ( H)   ,
\end{equation}
which implies 
\begin{equation}\label{conclusion of reduction}
   \Theta^\ast \ovl{ c}_{ e_{X} } = \frac{1}{2\kappa_L+1} \ovl{c}_{e_{\Sigma}  } .
\end{equation}

\begin{remark}
Moreover, the part concerning {\lag} {\specinv}s in the equations \eqref{diag left}, \eqref{diag right} imply the following stronger version of the equation \eqref{absolute spec inv descend}:
\begin{equation}\label{absolute and relative spec inv descend}
\begin{gathered}
    \ovl{c}_{e_X } (h \pi^\ast H)  = \ovl{\ell}_{\wt{L}} (h \pi^\ast H) = \frac{1}{2\kappa_L+1}  \ovl{\ell}_{L}(H) =  \frac{1}{2\kappa_L+1}  \ovl{c}_{e_\Sigma } (H) 
\end{gathered}
\end{equation}
for any {\hamil} $H$ on $\Sigma$. 
\end{remark}

We verify that the skeleton $\Delta$ satisfies Borman's \textit{smallness} condition. The properties of the idempotent $e_X$ from {\propo} \ref{Sanda} imply
\begin{equation}\label{suphv property}
    \ovl{c}_{e_X } =  \ovl{\ell}_{\wt{L}}
\end{equation}
which implies that $\wt{L}$ is $\ovl{c}_{e_X }$-{\suphv}. As $\wt{L} \subset X \backslash \Delta$, the set $X \backslash \Delta$ is also $\ovl{c}_{e_X }$-{\suphv}. Thus, from the criterion for smallness {\propo} \ref{criterion for smallness}, we see that the skeleton $\Delta$ is indeed satisfying Borman's \textit{smallness} condition.

We have seen that the skeleton $\Delta$ is small {\wrt} $\ovl{c}_{ e_{X} }$ so one can apply Borman's reduction to the {\EP} {\qmor} $\ovl{ c}_{ e_{X} }$ and obtain a {\qmor} on $\wt{\Ham} (\Sigma)$ which, according to \eqref{conclusion of reduction}, coincides with the {\EP} {\qmor} $\ovl{ c}_{ e_{\Sigma} }$ and this completes the proof of Theorem \ref{main thm intro} (a.k.a \ref{main thm}).

\end{proof}

\section{Examples}\label{examples}
In this section, we will see Example \ref{example intro}, to which Theorem \ref{main thm} apply, more in detail.

\subsection{Laurent and Novikov fields}

We start with a technical but a very useful and important remark concerning the {\coeff} field, which was considered in \cite[Section 4.2, 4.5]{[Kaw22]}. To summarize the point, to deal with {\specinv}s, e.g. {\EP} {\qmor}s, it is more convenient to work with the Laurent {\coeff}
$$\Lambda_{\text{Lau}} :=\{\sum_{k\geq k_0 } b_k t^k : k_0\in \mathbb{Z},b_k \in \mathbb{C} \}  ,   $$
 where $t$ is a formal variable, while {\lag} Floer theory is more suited to work with the universal Novikov field
$$\Lambda:=\{\sum_{j=1} ^{\infty} a_j T^{\lambda_j} :a_j \in \mathbb{C}, \lambda _j  \in \mathbb{R},\lim_{j\to +\infty} \lambda_j =+\infty \} .$$

In this paper, we have been working with the universal Novikov field $\Lambda$, as for example, {\propo} \ref{Sanda} requires the universal Novikov field $\Lambda$. However, with the universal Novikov field $\Lambda$, it becomes more complicated to consider {\EP} {\qmor}s, as $QH(X)=QH(X;\Lambda)$ is more complicated than $QH(X;\Lambda_{\text{Lau}})$, e.g. while $QH( \C P^n;\Lambda_{\text{Lau}})$ is a field, $QH( \C P^n;\Lambda)$ splits into a sum of $n+1$ fields:
$$QH( \C P^n;\Lambda) = \bigoplus_{1 \leq j \leq n+1} Q_j$$
where $Q_j$ is a field. Thus, it might give an impression that we get more {\EP} {\qmor}s by taking the universal Novikov field $\Lambda$ rather than the field of Laurent series (c.f. \cite[Remark 5.2]{[Wu15]}) but it was proved in \cite[Section 4.5 ]{[Kaw22]} that this is not the case. We will state the relevant result for the case of $\C P^n$ and $Q^n$ which we will use later in Section \ref{section of examples}.

\begin{prop}\label{comparison lemma}
\begin{enumerate}
    \item For $\C P^n$, $QH( \C P^n;\Lambda_{\text{Lau}})$ is a field and $QH( \C P^n;\Lambda)$ splits into a sum of $n+1$ fields but the {\EP} {\qmor}s all coincide:
    $$\ovl{c}_{1_X} =\ovl{c}_{e_X}$$
    for any unit of a field factor $e_X \in Q_j$.
    
    \item For $Q^n$, $QH( Q^n;\Lambda_{\text{Lau}})$ splits into a sum of two fields, i.e. 
    $$QH( Q^n;\Lambda_{\text{Lau}}) =Q_+ \oplus Q_- , $$
    and $QH(Q^n;\Lambda)$ splits into a sum of finer fields but the {\EP} {\qmor}s all coincide:
    $$\ovl{c}_{e_+} =\ovl{c}_{e_X} $$
    or
    $$\ovl{c}_{e_-} =\ovl{c}_{e_X}$$
    for any unit of a field factor $e_X \in Q_j$, depending on whether $Q_j$ is splitted from $Q_+ $ or $ Q_-$.
\end{enumerate}
\end{prop}

We do not give a proof as it is obtained by exactly the same argument as in \cite[Proof of Theorem 6, Remark 44]{[Kaw22]}.

\subsection{Examples}\label{section of examples}

\begin{enumerate}
    \item $(X,\Sigma)=(\C P^n, \C P^{n-1})$: The quantum cohomology rings of $\C P^n$ and $\C P^{n-1}$ with Laurent {\coeff} $\Lambda_{\text{Lau}}$ are both fields. Thus, according to Theorem \ref{EP quasimorphism}, their units $1_X$ and $1_\Sigma$ give rise to {\EP} {\qmor}s:
    
    \begin{equation}
\begin{gathered}
   \ovl{c}_{1_X} :  \wt{\Ham} (X) \to \R   \\
    \ovl{c}_{1_\Sigma}  : \wt{\Ham} (\Sigma) \to \R .
\end{gathered}
\end{equation}

Now, the skeleton for the pair $(\C P^n, \C P^{n-1})$ is a point which is \textit{small} {\wrt} $\ovl{c}_{1_X}$. Thus, Borman's reduction theorem \ref{Borman's reduction result} is applicable and $\Theta^\ast \ovl{c}_{1_X}$ is a {\qmor} for $\C P^{n-1}$, which we do not know at this point whether or not it is of {\EP}-type. The assumption of Theorem \ref{main thm} is satisfied for the pair of the Clifford tori $(\wt{L}=T^n _{\text{Clif}}$, $L=T^{n-1} _{\text{Clif}})$ whose {\suppot}s are as follows (see \cite{[Cho04]}):
     \begin{equation}
\begin{gathered}
  W_{T^{n} _{\text{Clif}}}= z_1 + z_2 + \cdots + z_n + \frac{1}{z_1 z_2 \cdots z_n}, \\
  W_{T^{n-1} _{\text{Clif}}}=z_1 + z_2 + \cdots + z_{n-1} + \frac{1}{z_1 z_2 \cdots z_{n-1}}.
\end{gathered}
\end{equation}
 These do have non-{\degen} critical points and by choosing non-{\degen} critical points as the local systems to setup Floer/pearl theory, we get units of field factors from the open-closed map by {\propo} \ref{Sanda}:
  \begin{equation}
\begin{gathered}
  e_{X}:= \frac{1}{\langle p_{\wt{L}}, p_{\wt{L}} \rangle}_{\text{Muk}} \mathcal{OC}^0 (pt_{\wt{L}})   , \\
  e_{\Sigma} := \frac{1}{\langle p_{K}, p_{K} \rangle}_{\text{Muk}} \mathcal{OC}^0 (pt_{L})   
\end{gathered}
\end{equation}
 
where $e_{X}$ (resp. $e_{\Sigma}$) is a unit factor in $QH(\C P^{n},\Lambda)$ (resp. $QH(\C P^{n-1},\Lambda)$). From {\propo} \ref{comparison lemma}, we have
  \begin{equation}
\begin{gathered}
  \ovl{c}_{1_X} = \ovl{c}_{e_X},\\
  \ovl{c}_{1_\Sigma} = \ovl{c}_{e_\Sigma}.
\end{gathered}
\end{equation}
 Thus, we obtain 
    $$\Theta^\ast \ovl{c}_{1_X} = \frac{1}{\kappa_\Sigma+1} \ovl{c}_{1_\Sigma} = \frac{n}{n+1} \ovl{c}_{1_\Sigma} . $$

    \item $(Q^n, Q^{n-1})$: The quantum cohomology rings of $Q^n$ and $Q^{n-1}$ with Laurent {\coeff} $\Lambda_{\text{Lau}}$ both split into direct sums of two fields:
     \begin{equation}
\begin{gathered}
   QH (Q^n )=QH (Q^n )_+ \oplus QH (Q^n )_- ,  \\
   QH (Q^{n-1} )=QH (Q^{n-1} )_+ \oplus QH (Q^{n-1} )_-. 
\end{gathered}
\end{equation}
    Thus, according to Theorem \ref{EP qmor}, the units of the field factors $e_{Q^n,\pm}$ and $e_{Q^{n-1},\pm}$ where 
        \begin{equation}
\begin{gathered}
   1_{Q^n}=e_{Q^n,+} + e_{Q^n,-},\\
   1_{Q^{n-1}}=e_{Q^{n-1},+} + e_{Q^{n-1},-},
\end{gathered}
\end{equation}
    give rise to {\EP} {\qmor}s:
    \begin{equation}
\begin{gathered}
   \ovl{c}_{e_{Q^n,\pm}} :  \wt{\Ham} (Q^n) \to \R  , \\
    \ovl{c}_{e_{Q^{n-1},\pm}}  : \wt{\Ham} (Q^{n-1}) \to \R .
\end{gathered}
\end{equation}
Now, the skeleton for the pair $(Q^n, Q^{n-1})$ is a {\lag} sphere 
$$ S^n  := \{z  \in \C P^{n+1}| z_0 ^2 + \cdots+z_{n} ^2+ z_{n+1} ^2 =0 ,\ z_0,\cdots,z_n \in \R,\ z_{n+1} \in i \R \} .$$
and it was proven in \cite[Theorem B]{Kaw23} that $Q^n \backslash S^n$ is $\ovl{c}_{e_{Q^n,+}}$-{\suphv} but not for $\ovl{c}_{e_{Q^n,-}}$-{\suphv}. Thus, by the criterion for smallness {\propo} \ref{criterion for smallness}, we see that $S^n$ is small {\wrt} $\ovl{c}_{e_{Q^n,+}}$ but not {\wrt} $\ovl{c}_{e_{Q^n,-}}$. Thus, Borman's reduction theorem \ref{Borman's reduction result} is applicable only to $\ovl{c}_{e_{Q^n,+}}$ and $\Theta^\ast \ovl{c}_{e_{Q^n,+}} $ is a {\qmor} for $Q^{n-1}$, which we do not know at this point whether or not it is of {\EP}-type. The assumption of Theorem \ref{main thm intro} is satisfied for the pair of monotone {\lag} tori $(\wt{L}=T^n _{\text{GZ}}$, $L=T^{n-1} _{\text{GZ}})$ (see Remark \ref{Biran and toric}). Their {\suppot}s are as follows (see \cite{[Kim]}):
    \begin{equation}
        \begin{gathered}
          W_{\wt{L}}(\wt{z}) = \frac{1}{z_{n}} + \frac{z_{n}}{z_{n-1}} + \cdots + \frac{z_{2}}{z_{1}} + 2z_{2} +z_1 z_2 ,\\
         W_{L}(z) = \frac{1}{z_{n-1}} + \frac{z_{n-1}}{z_{n-2}} + \cdots + \frac{z_{2}}{z_{1}} + 2z_{2} +z_1 z_2 .
        \end{gathered}
    \end{equation}

     These do have non-{\degen} critical points and by choosing non-{\degen} critical points as the local systems to setup Floer/pearl theory, we get units of field factors from the open-closed map by {\propo} \ref{Sanda}:
  \begin{equation}
\begin{gathered}
  e_{X}:= \frac{1}{\langle p_{\wt{L}}, p_{\wt{L}} \rangle}_{\text{Muk}} \mathcal{OC}^0 (pt_{\wt{L}})   , \\
  e_{\Sigma} := \frac{1}{\langle p_{L}, p_{L} \rangle}_{\text{Muk}} \mathcal{OC}^0 (pt_{L})   
\end{gathered}
\end{equation}
 
where $e_{X}$ (resp. $e_{\Sigma}$) is a unit factor in $QH(Q^{n};\Lambda_{\text{Lau}})_+$ (resp. $QH(Q^{n-1};\Lambda_{\text{Lau}})_+$). From {\propo} \ref{comparison lemma}, we have
     \begin{equation}
\begin{gathered}
  \ovl{c}_{e_{Q^n,+}} = \ovl{c}_{e_X},\\
  \ovl{c}_{e_{Q^{n-1},+}} = \ovl{c}_{e_\Sigma}.
\end{gathered}
\end{equation}
    
    Thus, we obtain
    $$\Theta^\ast \ovl{c}_{e_{Q^n,+}} = \frac{1}{\kappa_\Sigma+1} \ovl{c}_{e_{Q^{n-1},+}} = \frac{n-1}{n} \ovl{c}_{e_{Q^{n-1},+}} .$$

\begin{remark}\label{Biran and toric}
The monotone {\lag} torus $T^n _{\text{GZ}}$ in $Q^n$ that we use here was obtained by {\NNU} \cite{[NNU]} (see also Yoosik Kim \cite{[Kim]}) by considering a Gelfand--Zeitlin system via {\tordeg}. As the compatibility of {\tordeg} and Biran decomposition is not obvious, it is not obvious that $T^n _{\text{GZ}}$ in $Q^n$ coincides with the torus obtained by the Biran circle bundle construction to $T^{n-1} _{\text{GZ}}$ in $Q^{n-1}$ for the {\polar} $(Q^n,Q^{n-1})$, i.e. 
$$T^{n} _{\text{GZ}} =\wt{T^{n-1} _{\text{GZ}}} . $$
However, this compatibility was proven by the author in \cite[Theorem B]{Kaw23}.
\end{remark}

    \item $(\C P^3, Q^{2})$: With Laurent {\coeff} $\Lambda_{\text{Lau}}$,
    the quantum cohomology rings of $\C P^3$ is a field and of $Q^{2}$ splits into a direct sum of two fields, where the unit $1_\Sigma$ splits as
    $$1_\Sigma=e_{\Sigma,+} + e_{\Sigma,-} .$$
    Thus, according to Theorem \ref{EP qmor}, the units $1_X$, $e_{\Sigma,+}$ and $e_{\Sigma,-}$ give rise to {\EP} {\qmor}s:
    
    \begin{equation}
\begin{gathered}
   \ovl{c}_{1_X} :  \wt{\Ham} (X) \to \R   \\
    \ovl{c}_{e_{\Sigma,\pm}}  : \wt{\Ham} (\Sigma) \to \R .
\end{gathered}
\end{equation}

The skeleton for the pair $(\C P^3, Q^{2})$ is $ \R P^3 $. We know that the {\lag} torus $\wt{L}=T^3 _{\text{Ch}} $ is $\ovl{c}_{1_X}$-{\suphv} and as $\wt{L}=T^3 _{\text{Ch}} \subset \C P^3 \backslash \R P^3$, the set $\C P^3 \backslash \R P^3$ is also $\ovl{c}_{1_X}$-{\suphv}. Thus, by the criterion for smallness {\propo} \ref{criterion for smallness}, $\R P^3$ is small {\wrt} $\ovl{c}_{1_X}$. Thus, Borman's reduction theorem \ref{Borman's reduction result} is applicable to $\mu_{\C P^3;EP}$ 
    and $\Theta^\ast \mu_{\C P^3;EP} $ is a {\qmor} for $Q^{2}$, which we do not know at this point whether or not it is of {\EP}-type. The assumption of Theorem \ref{main thm intro} is satisfied for the pair of {\lag} tori $(\wt{L}=T^3 _{\text{Ch}}$, $L=T^{2} _{\text{Ch}})$ whose {\suppot}s were computed by Oakley--Usher \cite[Proof of {\cor} 8.6]{[OU16]} and Auroux \cite[{\cor} 5.13]{[Aur07]}, respectively, as follows:
    
    \begin{equation}
    \begin{aligned}
        W_{T^3 _{\text{Ch}}}& = \frac{1}{z_3} (z_1+z_1z_2 ^{-1} + z_1 ^{-1} z_2 +z_1 ^{-1} ) +z_3 ,\\
       W_{T^2 _{\text{Ch}}}&=z_1+z_1z_2 ^{-1} + z_1 ^{-1} z_2 +z_1 ^{-1} .
    \end{aligned}
    \end{equation}
    They have non-{\degen} critical points and by choosing non-{\degen} critical points as the local systems to setup Floer/pearl theory, we get units of field factors from the open-closed map by {\propo} \ref{Sanda}:
  \begin{equation}
\begin{gathered}
  e_{X}:= \frac{1}{\langle p_{\wt{L}}, p_{\wt{L}} \rangle}_{\text{Muk}} \mathcal{OC}^0 (pt_{\wt{L}})   , \\
  e_{\Sigma} := \frac{1}{\langle p_{L}, p_{L} \rangle}_{\text{Muk}} \mathcal{OC}^0 (pt_{L})   
\end{gathered}
\end{equation}
 
where $e_{X}$ (resp. $e_{\Sigma}$) is a unit factor in $QH(\C P^{3},\Lambda)$ (resp. $QH(Q^{n-1},\Lambda)$). From {\propo} \ref{comparison lemma}, we have
  \begin{equation}
\begin{gathered}
  \ovl{c}_{1_X} = \ovl{c}_{e_X},\\
  \ovl{c}_{1_\Sigma} = \ovl{c}_{e_{\Sigma,+}}.
\end{gathered}
\end{equation}
 Thus, we obtain 
    $$\Theta^\ast \ovl{c}_{1_X} = \frac{1}{\kappa_\Sigma+1} \ovl{c}_{1_\Sigma} = \frac{1}{2} \ovl{c}_{e_{\Sigma,+}} . $$

\begin{remark}
One should be able to generalize this example to $(\C P^n, Q^{n-1})$ by considering the pair of the {\lag} tori $(\wt{L}=T^n _{\text{Ch}}$, $L=T^{n-1} _{\text{Ch}})$ but the {\suppot} for $\wt{L}=T^n _{\text{Ch}}$ in $\C P^n$ needs to be computed.
\end{remark}

\end{enumerate}

\section{Discussions}\label{Discussion}

\subsection{About the assumption}

It is possible that the assumption on the {\suppot} of $L$ in Theorem \ref{main thm intro} (a.k.a \ref{main thm}) automatically implies the assumption on the {\suppot} of $\wt{L}$. 
\begin{question}
If the {\suppot} $W_L$ of $L$ has a non-{\degen} {\critpt}, then does the {\suppot} $W_{\wt{L}}$ of $\wt{L}$ also have a non-{\degen} {\critpt}?
\end{question}

The author learned that the relation between $W_L$ and $W_{\wt{L}}$ is currently being studied \cite{[DTVW]}.

As we claimed in Remark \ref{rmk main thm}, we expect the following to hold:

\begin{claim}\label{ideal thm}
Let $(X,\Sigma)$ be a pair of a closed monotone {\symp} {\mfd} and a Donaldson divisor therein. Assume there exists an {\EP} {\qmor} for $X$
$$\mu_{X} ^{EP} : \wt{\Ham} (X) \longrightarrow \R  $$
for which the skeleton $\Delta$ is small, i.e. 
$$ \tau_{\mu_{X} ^{EP}} (\Delta)=0 .$$

Then, there exist
\begin{itemize}
    \item There exists a monotone {\lag} torus $L$ in $\Sigma$ whose {\suppot} $W_L$ has a non-degenerate critical point.
    \item The {\suppot} $W_{\wt{L}}$ of the lifted monotone {\lag} torus $\wt{L}$ in $X$ also has a  non-degenerate critical point.
\end{itemize}
\end{claim}

If Claim \ref{ideal thm} holds, then the Borman's question will be solved in the following ideal form.

\begin{conj}
Let $(X,\Sigma)$ be a pair of a closed monotone {\symp} {\mfd} and a Donaldson divisor therein. If there exists an {\EP} {\qmor} for $X$
$$\mu_{X} ^{EP} : \wt{\Ham} (X) \longrightarrow \R  $$
for which the skeleton $\Delta$ is small, then it's reduction to $\Sigma$ is also an {\EP} {\qmor}, i.e. there exists an {\EP} {\qmor} for $\Sigma$ 
$$\mu_{\Sigma} ^{EP} : \wt{\Ham} (\Sigma) \longrightarrow \R $$
that satisfies
$$   \Theta^\ast \mu_{X} ^{EP} = \mu_{\Sigma} ^{EP} .  $$
\end{conj}

\begin{remark}
\begin{enumerate}
    \item We should be able to get more examples if we have more data of the {\suppot}s. The study of the {\suppot} is beyond the scope of this paper.
    
    \item The present paper only deals with monotone {\symp} {\mfd}s and {\lag}s for technical reasons but the author expects that similar argument should work for {\symp} {\mfd}s without the monotonicity condition. 
\end{enumerate}
\end{remark}

        \subsection{The {\hamil} torus action case}

In \cite{[Bor13]}, Borman establishes a similar reduction of {\qmor}s on $\wt{\Ham} (X)$ for the case where the {\symp} {\mfd} $X$ admits a {\hamil} torus action and asks an analogous question to Question \ref{Borman's question} for this case. One approach this question is to establish a version of results in \cite{[Sch21]} with action filtration and do a similar argument as in this paper. Note that {\thm} \ref{main thm intro} is precisely the $k=1$ case of the {\hamil} torus action version problem, as the {\suphv} assumption is also satisfied by \eqref{suphv property}.

\appendix

\Addresses

\end{document}